\documentclass{stylefile}

\newif\ifshowdeletions
\showdeletionsfalse

\newcommand{\CH}[2][\relax]{
\def\test{#1}
\ifshowdeletions\ifx\test\relax\else\DEL{ #1}\fi\fi
{\color{magenta}{#2}}
}

\newcommand{\DEL}[1]{{\color{magenta}\sout{#1}}}

\journal{Theoretical Computer Science}

\usepackage[utf8]{inputenc}
\usepackage[T1]{fontenc}

\usepackage[inline]{enumitem}
\usepackage{multirow}
\usepackage{amssymb}
\usepackage{amsmath}
\usepackage{booktabs}
\usepackage{xcolor}
\usepackage[normalem]{ulem}
\usepackage{subcaption}
\usepackage{here}
\usepackage{url}
\usepackage{hyperref}

\newcommand{\grt}{{\ge\,}}

\newcommand{\NN}{\mathbb{N}}
\newcommand{\ZZ}{\mathbb{Z}}

\newcommand{\e}{\,|\,}
\renewcommand{\mod}{\;\textrm{mod}\;}

\begin{document}

\newtheorem{theorem}{Theorem}
\newtheorem{remark}{Remark}
\newtheorem{corollary}{Corollary}
\newtheorem{lemma}{Lemma}
\newtheorem{definition}{Definition}
\newtheorem{proposition}{Proposition}
\newtheorem{observation}{Observation}
\newtheorem{strategy}{Strategy}
\newproof{proof}{Proof}
\newtheorem{claim}{Claim}

\begin{frontmatter}

\title{Optimal Strategies for Static Black-Peg AB~Game With
Two and Three Pegs\ \tnoteref{prelim}}\tnotetext[prelm]
{A preliminary version of a part of this paper appeared
in the proceedings of the 11th International Conference
on Combinatorial Optimization and Applications, COCOA 2017,
Ed.: Gao, X., Du, H. Han, M., LNCS 10628, pp.~409-424.}

\author[Math]{Gerold J\"ager}
\ead{gerold.jaeger@math.umu.se}
\address[Math]{Department of Mathematics and Mathematical Statistics,
University of Ume{\aa}, SE-901-87 Ume{\aa}, Sweden}

\author[CS]{Frank Drewes}
\ead{frank.drewes@cs.umu.se}
\address[CS]{Department of Computer Science,
University of Ume{\aa}, SE-901-87 Ume{\aa}, Sweden}

\begin{abstract}
The AB~Game is a game similar to the popular game Mastermind. We
study a version of this game called Static Black-Peg AB~Game.
It is played by two players, the codemaker and the codebreaker.
The codemaker creates a so-called secret
by placing a color from a set of $c$ colors
on each of $p \leq c$ pegs, subject to the condition that every color is used at most once.
The codebreaker tries to determine the secret by asking questions, 
where all questions are given at once and each question is a possible secret.
As an answer the codemaker reveals the number of correctly placed colors
for each of the questions.
After that, the codebreaker only has one more try to determine the secret and
thus to win the game.

For given $p$ and $c$, our goal is to find the smallest number
$k$ of questions the codebreaker needs to win, regardless of the secret,
and the corresponding list of questions,
called a $(k+1)$-strategy.
We present  a $ (\lceil 4c/3 \rceil-1) $-strategy for $p=2$
for all $ c \ge 2 $, 
and a $ \lfloor (3c-1)/2 \rfloor $-strategy for $p=3$ for all $ c \ge 4 $
and show the optimality of both strategies, i.e., we prove that no
$(k+1)$-strategy for a smaller~$k$ exists.

\end{abstract}


\end{frontmatter}

\section{Introduction}
\label{sec-intro}

The AB~Game (also known as ``bulls and cows game'') 
is a game similar to the popular game Mastermind. 
Whereas the first one dates back more than a century,
the latter was invented by Meirowitz
in 1970. Mastermind has since turned out to have interesting applications 
in fields such as cryptography~\cite{FL12}, bioinformatics~\cite{GET11} 
and privacy protection~\cite{AG13}.
In both games a \emph{codemaker} and a
\emph{codebreaker} play against each other. The codemaker chooses a secret 
code by placing colors from a set of $c$ available colors on $p$ pegs. In the 
original version of the AB Game and Mastermind, respectively, 
$p=4$, $c=6$ and $p=4$, $c=10$, respectively, are fixed, but to make the 
computational and mathematical properties of the game interesting, at least one 
of these parameters must be made variable. The goal of the codebreaker is to 
discover the secret code by making a sequence of guesses until the secret 
has been found. Each guess is a possible secret. 
The corresponding answers of the codemaker consist of
black and white pegs, a black one for each peg of the
question which is correct in both position and color, and
a white one for each peg which is correct in
color but not in position. The goal of the codebreaker is to minimize the
number of questions needed to find the secret, i.e., to receive $p$ black pegs as the last answer. We call games of this kind \emph{codebreaking games}.

Mastermind can be turned into a decision problem, 
Mastermind Satisfiability, as 
follows: given a list of questions and corresponding answers, are these 
answers compatible with at least one possible secret? Interestingly, 
in~\cite{SZ06} this problem was shown to be $\mathcal{NP}$-complete. 
For further results on (non-static) Mastermind see for 
example~\cite{DDST16,JP09,MS18}.

The difference between Mastermind and the AB~Game is that a possible secret or 
question of the latter may contain repeated colors whereas the AB~Game requires 
the secret as well as each question to consist of $p$ \emph{distinct} colors. Thus, the 
distinction between the two games mirrors the well-known distinction made in 
combinatorics between an ordered choice of $p$ items from a set of $c$ elements 
either with or without replacement.

Both Mastermind and the AB Game also have a black-peg 
variant, where the answers of the codebreaker contain only black pegs 
(i.e., in addition to how many pegs are correct in both position and 
color, no further information 
about correct colors placed on the wrong pegs is provided).

While strategies for the Black-Peg AB~Game were previously studied 
in~\cite{EGSS18,JP15a,KT86}, 
we continue our study of optimal strategies for the \emph{static} 
black-peg variant of codebreaking games which was started 
in~\cite{GJSS17,GJSS21,Jag16,JD18,JD20}. 
The static black-peg variant differs from the 
standard version described as follows. The codebreaker is 
required to present all questions except the last one (which reveals the secret) 
right at the beginning of the game. Thus, the game is \emph{static}, meaning 
that the codebreaker cannot adapt later questions to previous answers. 

In~\cite{DDST16} it was shown that the original and the static version of
Mastermind need $\mathcal{O}(n \log\log n)$ questions 
for the most prominent case $n=c=p$.
For the Static Black-Peg AB Game, for this case $n=p=c$,
a lower bound of $ \Omega(n\log n) $ and 
an upper bound of $ \mathcal{O}(n^{1.525}) $ were presented in~\cite{GJSS21}, 
and this upper bond has recently been improved to
$ \mathcal{O}(n \log n) $~\cite{LMS22}.

A \emph{$(k+1)$-strategy} is
a sequence of $k$ questions such that the answers to these questions uniquely
determine the secret (i.e., the codebreaker can win the game with
the $(k+1)$-th question). We are interested in 
optimal strategies -- $(k+1)$-strategies where $k+1$ is as small as possible.
Let us denote this optimal $k+1$ (depending on $p$ and $c$) by $ s(p,c) $. Erdös and 
Rényi~\cite{ER:63} and Söderberg and Shapiro~\cite{SS63}
showed independently that $ s(p,c) \in \mathcal{O}(p/{\log p}) $ for $c=2$, a 
result that was later generalized to $c \le p^{1-\epsilon} $ for $ \epsilon > 0 $ 
by Chv\'atal~\cite{Chv83}.
Goddard~\cite{God03} developed a $ (\left\lceil 2c/3 \right\rceil 
+1)$-strategy  for two pegs and $c$-strategies for three and 
four pegs. For sufficiently large $c$, these strategies are optimal.

In~\cite{Jag16}, we presented an optimal $\left\lceil 
(4c-1)/3 \right\rceil $-strategy  for Static Black-Peg Mastermind in the case of 
$p=2$ pegs and an arbitrary number $c \ge 1$ of colors, 
and in~\cite{JD20} an optimal $(\left\lfloor 
3c/2 \right\rfloor +1) $-strategy for Static Black-Peg Mastermind in the case of 
$p=3$ pegs and an arbitrary number $c \ge 2$ of colors. We now continue this 
line of research by considering the corresponding cases of the 
Static Black-Peg AB~Game with $p=2$ and $p=3$, respectively.
We start with an investigation of the simpler $p=2$ case 
and an arbitrary number $c$ of colors
(where by definition of the AB Game, $c \ge 2 $ must hold) 
resulting in an optimal 
$(\lceil 4c/3 \rceil-1)$-strategy. This strategy improves an earlier optimal 
strategy that was presented in~\cite{GJSS17} without explicit proof of 
feasibility and optimality. 
While our new strategy obviously has the same number 
of questions as the earlier one, the improvement lies in its structural 
simplicity. 

Our main result of this work is an optimal 
$ \lfloor (3c-1)/2 \rfloor $-strategy 
for the Static Black-Peg AB~Game with $p=3$
and an arbitrary number $c \ge 4$ of colors.
The strategies for both cases for the Static Black-Peg AB Game with $p=2$ and $p=3$ 
have a similar structure. This structure can also be applied to and simplify 
our earlier strategies for Static Black-Peg Mastermind with $p=2$ and 
$p=3$~\cite{Jag16,JD20}. 
Our expectation is that extending and comparing the strategies for Static Black-Peg Mastermind 
of~\cite{Jag16,JD20} and of the AB Game in this work will eventually make it 
possible to develop generic optimal 
strategies for arbitrary $p$ for both of these codebreaking games.

Interestingly, there is also a graph-theoretic interpretation of strategies
for Static Black-Peg Mastermind to the so-called metric dimension. 
For an undirected graph $G$, the metric dimension of $G$ 
is defined as the minimum number of vertices 
in a subset $S$ of $G$ such that all other vertices are uniquely determined by 
their distances to the vertices in $S$, see  
for example~\cite{FHHMS15,RKYS15,YKR11} for results on special graph classes.
Furthermore, the decision variant of this problem is $\mathcal{NP}$-complete~\cite{GJ79}

Concretely, Static Black-Peg Mastermind with $p$ pegs and $c$ colors 
needing exactly $k$ questions is equivalent to the metric dimension of the graph 
$ \ZZ_c^p $ being $k$ 
(see \cite{JD20} for more details). Thus, the results of~\cite{Jag16,JD20}
show that the metric dimension of $ \ZZ_n \times \ZZ_n $ is 
$\left\lceil (4c-1)/3 \right\rceil -1 $ and that of $ \ZZ_n \times \ZZ_n \times \ZZ_n $
is $\left\lfloor 3c/2 \right\rfloor $.
An open question is whether the results of the present paper can also be connected
to the metric dimension of other types of graphs. In any case, our new way to 
assemble optimal strategies by iterating color-shifted copies of a fixed block 
can similarly be used to compose simpler optimal strategies for Static Black-Peg 
Mastermind than those in~\cite{Jag16,JD20}. In turn, this results in simplified 
proofs of the fact that the metric dimensions of $ \ZZ_n \times \ZZ_n $ and $ 
\ZZ_n \times \ZZ_n \times \ZZ_n $ are $\left\lceil (4c-1)/3 \right\rceil -1 $ 
and $\left\lfloor 3c/2 \right\rfloor $, respectively.

This work is set up as follows. After giving some
preliminaries in Section~\ref{sec-prel},
we present an optimal $ (\lceil 4c/3 \rceil-1) $-strategy for $p=2$ in
Section~\ref{sec_two}
and an optimal $ \lfloor (3c-1)/2 \rfloor $-strategy for $p=3$ in
Section~\ref{sec_three}.
Finally, we conclude and suggest possible future work
in Section~\ref{sec-futurework}.

\section{Preliminaries}
\label{sec-prel}

Let $p\ge 1$ denote the number of pegs and $c \ge p$ the number of colors.
W.l.o.g., throughout the remainder of this paper, let the 
pegs be numbered by $ 1,2\dots,p$, and the colors by $ 1,2,\dots,c$. 
We write the questions and secrets in the 
form 	$ Q = (q_1\e q_2\e \dots\e q_p)$.
The possible answers are written as 
$0$B, $1$B, $2$B, $\dots$, $p$B. 

For $k\in\NN$, a $(k+1)$-\emph{strategy} 
for Static Black-Peg AB~Game
consists of $k$ questions which the codebreaker has to 
ask altogether at the beginning of the game. 
These are the so-called \emph{main questions.}
Such a strategy is \emph{feasible} if
every secret $S$ is uniquely determined 
by the $k$ answers.  
Having received these answers, the codebreaker 
can ask the \emph{final question} $S$ 
to win the game.

We use the letter $k$ to denote the number of main questions, excluding
the final question. Since we will only be concerned with the main questions, in the following
we will generally omit the term ``main'', referring to the main questions
simply as questions. A $(k+1)$-strategy is called
\emph{optimal} if there is no feasible $k$-strategy.
Clearly, all questions of an optimal strategy must be distinct. 
Therefore, we shall in the following only consider strategies 
in which all questions are distinct, without explicitly mentioning 
this fact. 

\begin{remark}
Determining an optimal strategy for the Static Black-Peg AB~Game is trivial
in the case $p=1$ because the  questions $(1),(2),\dots,(c-1)$ 
reveal the secret and, obviously, no set consisting of fewer than $c-1$ questions does.
Thus, in the case $p=1$ this strategy is an optimal strategy for all $c$ (and so is
any other strategy consisting of $c-1$ questions).
In other words, there is an optimal $c$-strategy for $p=1$ 
and every $c\ge1$.
\end{remark}

We need the following definition.

\begin{definition}
\label{def1}
Let $Q=(q_1\e q_2\e \dots\e q_p)$ and $Q'=(q_1'\e q_2'\e \dots\e q_p')$ be questions.
\begin{enumerate}[label=(\alph*)]
\item
\label{def1a}
$Q$ and $Q'$ are  
called \emph{neighboring}, if $q_i=q_i'$ for some $i\in\{1,2,\dots,p\}$.
We say that they \emph{overlap} in peg $i$
and call peg $i$ an \emph{overlapping peg} (of $Q$ and $Q'$).
\item
\label{def1b}
$Q$ and $Q'$ are are \emph{double neighboring} if they overlap in at least two distinct pegs. 
\item 
\label{def1c}
$Q$ and $Q'$ are \emph{disjoint}, if $q_i\neq q_j'$ for all 
$i,j\in\{1,2,\dots,p\}$. More generally, we say that $Q$ and $Q'$ are disjoint 
in pegs $i_1,i_2,\dots,i_k$ if $\{q_{i_1},q_{i_2},\dots,q_{i_k}\}
\cap\{q_{i_1}',q_{i_2}'\dots,q_{i_k}'\}=\emptyset$.
\item 
\label{def1d}
For $ a_1,a_2,\dots,a_p \in \NN $, $Q$ is a
\emph{$(a_1,a_2,\dots,a_p)$-question}\footnote{Note 
the different notation 
in comparison to the question itself,
where we separate the numbers by the
symbol ``$\e$''.}
of a given strategy if the $i$-th color of $Q$ occurs 
exactly $a_i$ times on the $i$-th peg of the questions
of the strategy, for $i=1,2,\dots,p$.
Sometimes, for one or more $i\in\{1,2,\dots,p\}$, we do not want to specify $a_i$, 
i.e., it is not relevant how often the $i$-th color occurs on the $i$-th peg. Then
``$a_i$'' is replaced by the symbol ``$\star$''.
Finally, ``$a_i$'' can be replaced by ``$\ge a_i$'' to express that
$q_i$ occurs \emph{at least}
$a_i$ times on the $i$-th peg.
\end{enumerate}
\end{definition}

In the remainder of the paper, we investigate the cases $p=2,3$.
For both cases, the feasible strategy starts with so-called base questions for a 
small number of colors and then repeats
a copy of a fixed block of questions,
shifting the colors in each copy appropriately.
This structure makes the proof of feasibility relatively easy. However,
for the proof of optimality all possible feasible strategies have to be considered, without making specific assumptions regarding their structure, and it has to be excluded that they may require fewer questions.

\section{Two Pegs}
\label{sec_two}

In this section let $p=2$.

\subsection{A $ (\lceil 4c/3 \rceil-1) $-Strategy}
\label{ssec-two-pegs-strat}

We introduce a $ (\lceil 4c/3 \rceil-1) $-strategy for 
each $c\ge 2$ which we will later show to be feasible
and optimal. We start with presenting such a 
strategy explicitly for $c=2,3,4$ 
in Table~\ref{tab234}.
These three strategies have been found by 
brute-force computer search.

\begin{table}
\setlength{\tabcolsep}{2mm}
\centering
\begin{subtable}[b]{9em}
\begin{tabular}{c || c | c}
Peg & $1$ & $2$ 
\\
\hline
\hline
$Q_1$ & $1$ & $2$ 
\end{tabular}
\subcaption{$c=2$, $k=1$.}
\label{tab234a}
\end{subtable}
\hspace*{0.5em}
\begin{subtable}[b]{9em}
\begin{tabular}{c || c | c}
Peg & $1$ & $2$ 
\\
\hline
\hline
$Q_1$ & $1$ & $2$ 
\\
\hline
$Q_2$ & $3$ & $1$ 
\end{tabular}
\subcaption{$c=3$, $k=2$.}
\label{tab234b}
\end{subtable}
\hspace*{0.5em}
\begin{subtable}[b]{9em}
\begin{tabular}{c || c | c}
Peg & $1$ & $2$ 
\\
\hline
\hline
$Q_1$ & $1$ & $3$ 
\\
\hline
$Q_2$ & $3$ & $1$ 
\\
\hline
$Q_3$ & $2$ & $3$ 
\\
\hline
$Q_4$ & $3$ & $2$ 
\end{tabular}
\subcaption{$c=4$, $k=4$.}
\label{tab234c}
\end{subtable}
\caption{Feasible and optimal 
$ (\lceil 4c/3 \rceil-1) $-strategies 
for $ p=2$ and $ 2 \le c \le 4 $.}
\label{tab234}
\end{table}

As mentioned, we create 
$ (\lceil 4c/3 \rceil-1) $-strategies 
for all $ c \ge 2$ by starting with a
block of so-called base questions and then repeatedly
adding copies of a fixed block of questions,
shifting the colors in each copy appropriately. 
As for the base questions, we start with one of the three 
strategies of Table~\ref{tab234}.
The iterative step is based on the strategy
of Table~\ref{tab234c}, which we repeat a suitable
number of times (with shifted colors).

Concretely, let $ c \ge p=2 $ and 
$ c = 3s + t $, where $ s \in \NN_0 $, 
$ t \in \{2,3,4\}$ are uniquely determined.
For $ t=2$ we start with the $k=1$ question 
of the strategy of Table~\ref{tab234a},
for $ t=3 $ with 
the $k=2$ questions of the strategy of 
Table~\ref{tab234b} and for $ t=4 $ with 
the $k=4$ questions of the strategy of 
Table~\ref{tab234c}. 
As mentioned, we call this first group of 
questions the \emph{base questions}.

In all three cases the base questions
are followed $s $ times by 
the $4$ questions of Table~\ref{tab234c}. 
For $l=1,2,\dots,s $, 
the number $t+3\left(l-1\right)$ 
is added to the colors of the questions of
Table~\ref{tab234c}.
We call this second group of
questions the \emph{iterated questions}.

As examples, we present the corresponding
$ (\lceil 4c/3 \rceil-1) $-strategies for 
$5 \le c \le 10$ explicitly in Table~\ref{tab2510}.

\begin{table}
\setlength{\tabcolsep}{2mm}
\centering
\begin{subtable}[b]{9em}
\begin{tabular}{c || c | c}
Peg & $1$ & $2$ 
\\
\hline
\hline
$Q_1$ & $1$ & $2$ 
\\
\hline
\hline
\hline
$Q_2$ & $3$ & $5$ 
\\
\hline
$Q_3$ & $5$ & $3$ 
\\
\hline
\hline
$Q_4$ & $4$ & $5$ 
\\
\hline
$Q_5$ & $5$ & $4$ 
\end{tabular}
\subcaption{$c=5$, $k=5$.}
\label{tab2a}
\end{subtable}
\hspace*{0.5em}
\begin{subtable}[b]{9em}
\begin{tabular}{c || c | c}
Peg & $1$ & $2$ 
\\
\hline
\hline
$Q_1$ & $1$ & $2$ 
\\
\hline
$Q_2$ & $3$ & $1$ 
\\
\hline
\hline
\hline
$Q_3$ & $4$ & $6$ 
\\
\hline
$Q_4$ & $6$ & $4$ 
\\
\hline
\hline
$Q_5$ & $5$ & $6$ 
\\
\hline
$Q_6$ & $6$ & $5$ 
\end{tabular}
\subcaption{$c=6$, $k=6$.}
\label{tabb2b}
\end{subtable}
\hspace*{0.5em}
\begin{subtable}[b]{9em}
\begin{tabular}{c || c | c}
Peg & $1$ & $2$ 
\\
\hline
\hline
$Q_1$ & $1$ & $3$ 
\\
\hline
$Q_2$ & $3$ & $1$ 
\\
\hline
$Q_3$ & $2$ & $3$ 
\\
\hline
$Q_4$ & $3$ & $2$ 
\\
\hline
\hline
\hline
$Q_5$ & $5$ & $7$ 
\\
\hline
$Q_6$ & $7$ & $5$ 
\\
\hline
\hline
$Q_7$ & $6$ & $7$ 
\\
\hline
$Q_8$ & $7$ & $6$ 
\end{tabular}
\subcaption{$c=7$, $k=8$.}
\label{tabb2c}
\end{subtable}
\vspace{1.1em}
\begin{subtable}[b]{9em}
\begin{tabular}{c || c | c}
Peg & $1$ & $2$ 
\\
\hline
\hline
$Q_1$ & $1$ & $2$ 
\\
\hline
\hline
\hline
$Q_2$ & $3$ & $5$ 
\\
\hline
$Q_3$ & $5$ & $3$ 
\\
\hline
\hline
$Q_4$ & $4$ & $5$ 
\\
\hline
$Q_5$ & $5$ & $4$ 
\\
\hline
\hline
\hline
$Q_6$ & $6$ & $8$ 
\\
\hline
$Q_7$ & $8$ & $6$ 
\\
\hline
\hline
$Q_8$ & $7$ & $8$ 
\\
\hline
$Q_9$ & $8$ & $7$ 
\end{tabular}
\subcaption{$c=8$, $k=9$.}
\label{tabb2d}
\end{subtable}
\hspace*{0.5em}
\begin{subtable}[b]{9em}
\begin{tabular}{c || c | c}
Peg & $1$ & $2$ 
\\
\hline
\hline
$Q_1$ & $1$ & $2$ 
\\
\hline
$Q_2$ & $3$ & $1$ 
\\
\hline
\hline
\hline
$Q_3$ & $4$ & $6$ 
\\
\hline
$Q_4$ & $6$ & $4$ 
\\
\hline
\hline
$Q_5$ & $5$ & $6$ 
\\
\hline
$Q_6$ & $6$ & $5$ 
\\
\hline
\hline
\hline
$Q_7$ & $7$ & $9$ 
\\
\hline
$Q_8$ & $9$ & $7$ 
\\
\hline
\hline
$Q_9$ & $8$ & $9$ 
\\
\hline
$Q_{10}$ & $9$ & $8$ 
\end{tabular}
\subcaption{$c=9$, $k=10$.}
\label{tabb2e}
\end{subtable}
\hspace*{0.5em}
\begin{subtable}[b]{9em}
\begin{tabular}{c || r | r}
Peg & $1$ & $2$ 
\\
\hline
\hline
$Q_1$ & $1$ & $3$ 
\\
\hline
$Q_2$ & $3$ & $1$ 
\\
\hline
$Q_3$ & $2$ & $3$ 
\\
\hline
$Q_4$ & $3$ & $2$ 
\\
\hline
\hline
\hline
$Q_5$ & $5$ & $7$ 
\\
\hline
$Q_6$ & $7$ & $5$ 
\\
\hline
\hline
$Q_7$ & $6$ & $7$ 
\\
\hline
$Q_8$ & $7$ & $6$ 
\\
\hline
\hline
\hline
$Q_9$ & $8$ & $10$ 
\\
\hline
$Q_{10}$ & $10$ & $8$ 
\\
\hline
\hline
$Q_{11}$ & $9$ & $10$ 
\\
\hline
$Q_{12}$ & $10$ & $9$ 
\end{tabular}
\subcaption{$c=10$, $k=12$.}
\label{tabb2f}
\end{subtable}
\caption{Feasible and optimal 
$ (\lceil 4c/3 \rceil-1) $-strategies 
for $ p=2$ and $ 5 \le c \le 10 $.}
\label{tab2510}
\end{table}

In Subsections~\ref{ssec-two-pegs-feas}
and~\ref{ssec-two-pegs-opt} we prove
the following theorem.

\begin{theorem}
\label{thp2}
The presented strategy 
is a \emph{feasible} and \emph{optimal} 
$ (\lceil 4c/3 \rceil-1) $-strategy 
for $p=2$ and for the corresponding $c \ge 2 $.
\end{theorem}

In the following let 
$ h \equiv c\ \! \mod 3$.

\begin{remark}
\begin{enumerate}[label=(\alph*)]
\item For $ h\in\{0,2\}$,
our presented strategy 
for the Static Black-Peg AB~Game 
needs \emph{one question less} than the strategy
of~\cite{Jag16} for Static Black-Peg Mastermind
(see Strategies~1 and~3, respectively, in~\cite{Jag16}).
\item For $ h=1$,
our presented strategy for the Static Black-Peg 
AB~Game needs \emph{the same number of questions} 
as the strategy of~\cite{Jag16} 
for Static Black-Peg Mastermind
(see Strategy~2 in~\cite{Jag16}).
\end{enumerate}
\end{remark}

\subsection{Feasibility of the $ (\lceil 4c/3 \rceil-1) $-Strategy}
\label{ssec-two-pegs-feas}

Analyzing the $ (\lceil 4c/3 \rceil-1) $-strategy, we observe the
following:

\begin{observation}
\label {obs1}
\begin{enumerate}[label=(\alph*)]
\item For $h=0$, the following holds:

The $ (\lceil 4c/3 \rceil-1) $-strategy
contains only $ (1,2) $-questions and $ (2,1) $-questions,
except the first two questions 
$ Q_1 = (1\e2) $ and $ Q_2 = (3\e 1) $ which are both $ (1,1)$-questions.

On the first peg only the color $2$ is missing
(throughout the entire strategy), 
and on the second peg only the color $3$ is missing.

\item For $h=1$, the following holds:

The $ (\lceil 4c/3 \rceil-1) $-strategy
contains only $ (1,2) $-questions and $ (2,1) $-questions.

Both on the first and the second peg only the color $4$ is missing.

\item For $h=2$, the following holds:

The $ (\lceil 4c/3 \rceil-1) $-strategy
contains only $ (1,2) $-questions and $ (2,1) $-questions,
except the first question $ Q_1 = (1\e 2) $ which is a $ (1,1)$-question.

On the first peg only the color $2$ is missing, 
and on the second peg only the color $1$ is missing.

\end{enumerate}
\end{observation}

Consider one fixed $ (1,2) $-question or $ (2,1)$-question, and
assume that it receives a non-empty answer $2$B or $1$B, respectively. 
Then the following conclusions can be drawn.
\begin{enumerate}[label=(\Roman*)]
\item The answer is $1$B.

Here we have to find out which one of the
two pegs is correct. This is determined by the answer
to the neighboring question: if that answer 
is also non-empty, then the color of the overlapping 
peg is correct, otherwise the color
of the other peg is correct.

As an example, consider the strategy
for $c=9$ in Table~\ref{tabb2e}. 

For the secret $ (4 \e 9) $, the answer
to the question $ Q_3 = (4 \e 6) $ is $1$B.
The neighboring question to $Q_3$ is $ Q_5 = 
(5 \e 6) $.
As $Q_5$ gives the answer $0$B, we know that color $4$ is
correct on the first peg, but we do not know yet 
the color of the second peg.

On the other hand, for the secret $ (3,6) $, the answer
to the question $ Q_3 $ is $1$B 
and to $ Q_5 $ also $1$B.
Then we know that color $6$ is correct on the second peg. 

\item The answer is $2$B.

The secret is found.
\end{enumerate}

So as soon we have a non-empty answer to a $ (1,2)$-question
or to a $ (2,1) $-question, we know the color of one peg.
However, having this information about one peg, also the other peg
can be determined from the answers to all questions of the strategy,
as by Observation~\ref{obs1}) on each peg only one color is missing. 

So we have shown the feasibility for the case that 
at least one $ (1,2) $-question or at least one 
$ (2,1)$-question receives a non-empty answer.
To show the feasibility also for the remaining case 
that all $ (1,2) $-questions and all 
$ (2,1)$-questions receive an empty answer,
we distinguish between $ h=0,1,2$.

\begin{enumerate}[label=(\Roman*)]
\item $ h= 0$:

Six possible secrets are consistent with
answers $0$B to all $ (1,2) $-questions and $ (2,1)$-questions. 
We list them and the combination of answers to 
the $ (1,1) $-questions $Q_1 = (1 \e 2) $ 
and $ Q_2 = (3 \e 1) $:

\begin{enumerate}[label=(\roman*)]
\item Secret $ (1 \e 2) $: Combination of answers  $ (2\mbox{B},0\mbox{B})$.
\item Secret $ (1 \e 3) $: Combination of answers  $ (1\mbox{B},0\mbox{B})$.
\item Secret $ (2 \e 1) $: Combination of answers  $ (0\mbox{B},1\mbox{B})$.
\item Secret $ (2 \e 3) $: Combination of answers  $ (0\mbox{B},0\mbox{B})$.
\item Secret $ (3 \e 1) $: Combination of answers  $ (0\mbox{B},2\mbox{B})$.
\item Secret $ (3 \e 2) $: Combination of answers  $ (1B,1B)$.
\end{enumerate}

\item $ h=1$:

Only one possible secret is consistent with
the answer $0$B to all $ (1,2) $- questions and $ (2,1)$-questions,
namely $ (4 \e 4) $.

\item $ h=2$:

Two possible secrets are consistent with
the answer $0$B to all $ (1,2) $- and $ (2,1)$-questions,
namely $ (1 \e 2) $ and $(2 \e 1)$. However, these are distinguished
by the only $ (1,1)$-question $ Q_1 = (1 \e 2) $.

\end{enumerate}

Thus, each possible secret is uniquely determined in all three cases,
which finishes the proof of feasibility.

\subsection{Optimality of the $ (\lceil 4c/3 \rceil-1) $-Strategy}
\label{ssec-two-pegs-opt}

We need the following lemma.

\begin{lemma}
\label{lem1}
For the Static Black-Peg AB~Game with $p=2$ the following statements hold:

\begin{enumerate}[label=(\alph*)]
\item
\label{item1a}
For each feasible strategy and for each peg there exists at most
one color which does not occur on this peg. 
\item
\label{item1b}
A feasible strategy cannot contain two disjoint
$ (1,1) $-questions.
\item 
\label{item1c}
If a feasible strategy contains 
three $(1,1)$-questions, by possibly permuting the colors,
these questions can be written in the form 
$ (1 \e 2) $, $ (2 \e 3) $, $ (3\e 1) $.
\item
\label{item1d}
If a feasible strategy contains 
three $(1,1)$-questions, 
then on at least one of the two pegs of the  
questions, all $c$ colors must occur.
\item
\label{item1e}
A feasible strategy cannot contain 
four $(1,1)$-questions. 
\end{enumerate}

\end{lemma}

\begin{proof}
\begin{enumerate}
\item
The assertion is clear for $c=2$.
Thus, let $ c \ge 3 $.
Assume without loss of generality that there are two colors $a$
and $b$ which do not occur on the first peg. Choose an arbitrary
color~$x$ which is neither $a$ nor $b$ ($x$ exists 
because $ c \ge 3$).  Then the 
possible secrets
$ (a\e x) $ and $ (b \e x) $ receive the same combination
of answers. Thus, the two possible secrets are indistinguishable,
contradicting the feasibility of the strategy.

By symmetry, this also proves the corresponding statement for the second peg.
\item[\ref{item1b}]
Suppose that the strategy contains two disjoint
$(1,1)$-questions. So their four colors
are distinct, say
$ (1\e 2) $ and $ (3 \e 4) $, and thus $ c \ge 4 $.
Then the possible secrets $ (1 \e 4) $ 
and $ (3 \e 2) $ receive the same answer ($1$B) 
for both questions $ (1 \e 2) $ and $ (3 \e 4) $, 
and also the same answer ($0$B) for all other questions. 
Thus, the two possible secrets are  
indistinguishable, leading to a contradiction. 
\item[\ref{item1c}]
Let a feasible strategy contain three
$(1,1)$-questions, say
$ (q_1 \e q_1') $, $ (q_2 \e q_2') $ and $ (q_3 \e q_3') $.
Without loss of generality, we may permute colors so that $ q_1 = 1$, $ q_2 = 2 $, $ q_3 = 3 $. (Note that $q_1,q_2,q_3$ are distinct by the definition of $ (1,1) $-questions, and so are $q_1',q_2',q_3'$.)

If some $q_i'$, say $q_1'$, differs from $1, 2$ and $3$,
then not both $q_2' $ and $q_3'$
can be equal to $1$. Thus,
in this case $(1 \e q_1')$ and $(2 \e q_2')$, or
$(1 \e q_1')$ and $(3 \e q_3')$ would be
disjoint. By~\ref{item1b}, this contradicts the feasibility of
the strategy, thus showing that $\{q_1',q_2',q_3'\}=\{1,2,3\}$.

As $q_i=i$, the definition of
the AB~Game yields $q_i'\neq i$, which means that
$$q_1'\in\{2,3\},\ q_2'\in\{1,3\},\ q_3'\in\{1,2\},$$
which leaves only two possibilities, namely the claimed
$$
  (q_1 \e q_1') = (1\e2),\ (q_2 \e q_2')=(2\e 3),\ (q_3 \e q_3')=(3\e1)
$$
or
$$
  (q_1 \e q_1') = (1\e3),\ (q_2 \e q_2')=(2\e 1),\ (q_3 \e q_3')=(3\e2).
$$
In the second case, the claimed situation is obtained by switching colors 
$2$ and $3$ and switching the second and the third question.

\item[\ref{item1d}]
By~\ref{item1c}, a feasible strategy
may be assumed to contain three
$(1,1)$-questions $ (1 \e 2) $, $ (2 \e 3) $, 
and $ (3 \e 1) $.
Suppose that there is a color $a$ which does not
occur on the first peg of the questions
and a color $b$ which does not
occur on the second peg of the questions.
Then the possible secrets $ (1 \e b) $ 
and $ (a \e 2) $
receive the answer $1$B for
the question $ (1 \e 2) $, and $0$B
for all other questions, i.e.,
the same combination of answers. Thus,
these possible secrets $ (1 \e b) $ and 
$ (a \e 2) $ are
indistinguishable -- a contradiction.

\item[\ref{item1e}]
If a feasible strategy contains four
$(1,1)$-questions, by~\ref{item1c} the first three
of them can be written in the form
$ (1 \e 2) $, $ (2 \e 3) $, and $ (3 \e 1) $.
Adding a fourth $ (1,1) $-question would automatically
lead to two disjoint $ (1,1) $-questions.
This contradicts~\ref{item1b}. 
\hfill $ \square $
\end{enumerate}
\end{proof}

Now we come to the final part of the optimality proof.
Note that the presented strategy
uses $k=\lceil 4c/3 \rceil-2$ questions.\footnote{Recall
that $k$ does not include the final question.}
In the following we prove that every feasible strategy
consists of at least $k$ questions.
Thus, consider any feasible strategy. 
Let $l_i$ be the number of colors which
occur exactly once on the $i$-th peg of this strategy
for $i=1,2$. Furthermore, let $m$ be the number of $(1,1)$-questions of the strategy.
Observe that the total number of questions is at least $l_1+l_2-m$.
Again we distinguish between $ h=0,1,2$.

\begin{enumerate}[label=(\Roman*)]
\item $ h=0$.

We have
\begin{eqnarray*}
k & = & \frac{4c}{3} -2.
\end{eqnarray*}
On the one hand, if $ l_i \le 2c/3 $ for some $i\in\{1,2\}$, since at least $c-1$ colors
occur on peg~$i$ (see Lemma~\ref{lem1}\ref{item1a}),
we have at least 
\begin{eqnarray*}
\frac{2c}{3} + 2 \cdot 
\left( \frac{c}{3}-1 \right) \; \ge \; 
\frac{4c}{3}-2 \; = \; k
\end{eqnarray*}
questions.

If, on the other hand, $ l_i \ge \frac{2c}{3}+1 $ for $i=1,2$ it follows that there are at least
\begin{eqnarray*}
l_1+l_2-m\;\ge\; 2 \cdot \left(\frac{2c}{3}+1\right) - 3
\; \ge \; \frac{4c}{3}-1\; = \; k+1 
\end{eqnarray*}
questions, because we know from Lemma~\ref{lem1}\ref{item1e}
that $m\le3$.

\item $ h=1$.

We have
\begin{eqnarray*}
k & = & \frac{4c-1}{3} -1.
\end{eqnarray*}

We distinguish between two sub-cases.

\begin{enumerate}[label=(\roman*)]
\item On both pegs only $c-1$ colors occur.

On the one hand, analogously to the case $h=0$,
if $ l_i \le \frac{2c-2}{3} $ for some $i\in\{1,2\}$,
we have at least 
\begin{eqnarray}
\label{prev_eq}
\frac{2c-2}{3} + 2 \cdot 
\left( \frac{c-1}{3}\right) \; 
\ge \; \frac{4c-1}{3}-1 \; = \; k
\end{eqnarray}
questions. 

If instead $ l_i \ge \frac{2c+1}{3} $ for $i=1,2$, we have
\begin{eqnarray*}
l_1+l_2-m\;\ge\; 2 \cdot \left(\frac{2c+1}{3}\right) - 2 \; \ge \; 
\frac{4c}{3} +\frac{2}{3} - \frac 63 \; = \; k,
\end{eqnarray*}
because we know from
Lemma~\ref{lem1}\ref{item1d} that $m\le 2$.

\item On at least one peg, say peg~$1$, all $c$ colors occur.

If $ l_1 \le \frac{2c+1}{3} $, then there are at least 
\begin{eqnarray*}
\frac{2c+1}{3} + 2 
\cdot \left( \frac{c-1}{3}\right) \; 
\ge \; \frac{4c-1}{3} \; = \; k+1
\end{eqnarray*}
questions.

Recall also from Eq.~\eqref{prev_eq} that
we have at least $k$ questions if $ l_2 \le \frac{2c-2}{3} $,
and assume thus that $ l_1 \ge \frac{2c+4}{3} $ and
$ l_2 \ge \frac{2c+1}{3} $. Since $m\le3$ by
Lemma~\ref{lem1}\ref{item1e}, it follows that
\begin{eqnarray*}
l_1+l_2-m\;\ge\;\frac{2c+4}{3} + \frac{2c+1}{3} - 3
\; = \; \frac{4c-1}{3} - 1\; = \; k.
\end{eqnarray*}
\end{enumerate}

\item $ h=2$.

We have
\begin{eqnarray*}
k & = & \frac{4c-2}{3} -1.
\end{eqnarray*}
Analogously to the previous cases, if $ l_i \le \frac{2c-1}{3} $
for some $i\in\{1,2\}$,
then we have at least 
\begin{eqnarray*}
\frac{2c-1}{3} + 2 \cdot 
\left( \frac{c-2}{3}\right) \; \ge \; 
\frac{4c-2}{3}-1 \; = \; k
\end{eqnarray*}
questions.

It remains to consider the case where
$ l_i \ge \frac{2c+2}{3} $ for $i=1,2$. Again using the fact that
$m\le3$ by Lemma~\ref{lem1}\ref{item1e}, it follows that
\begin{eqnarray*}
l_1+l_2-m\;\ge\; 2 \cdot \left(\frac{2c+2}{3}\right) - 3
\; = \; \frac{4c+4}{3} - 3\; = \; \frac{4c-2}{3} - 1 \;=\; k,
\end{eqnarray*}
which completes the last case and thus the proof. 
\hfill $ \square $
\end{enumerate}

\section{Three Pegs}
\label{sec_three}

In this section let $p=3$.

\subsection{A $ \lfloor (3c-1)/2 \rfloor $-Strategy}
\label{ssec-three-pegs-strat}

We introduce a $ \lfloor (3c-1)/2 \rfloor $-strategy 
for each $c\ge 4$ which we will later show to be feasible
and optimal. Interestingly, such a feasible strategy
does not exist for $c=3$, i.e., no feasible
$ 4 $-strategy exists~\footnote{Observe that 
$ (\lfloor (3 \cdot 3 -1)/2 \rfloor) = 4 $, but this
includes the final question.}
although there are only $ 3! = 6 $
possible secrets. One feasible and optimal $5$-strategy
for $c=3$ is shown in Table~\ref{tab31}.

\begin{table}
\setlength{\tabcolsep}{2mm}
\centering
\begin{tabular}{c || r | r | r}
Peg & $1$ & $2$ & $3$ 
\\
\hline
\hline
$Q_1$ & $1$ & $2$ & $3$ 
\\
\hline
$Q_2$ & $1$ & $3$ & $2$ 
\\
\hline
$Q_3$ & $2$ & $1$ & $3$ 
\\
\hline
$Q_4$ & $2$ & $3$ & $1$ 
\end{tabular}
\caption{Feasible and optimal strategy for $p=3$,  $c=3$, $k=4$.}
\label{tab31}
\end{table}

We start by explicitly presenting
$ \lfloor (3c-1)/2 \rfloor $-strategies for $c=4,5,6,7,8,9$ in Table~\ref{tab3}. 
These six strategies have been found by brute-force computer search.

\begin{table}
\setlength{\tabcolsep}{2mm}
\centering
\begin{subtable}[b]{9em}
\begin{tabular}{c || c | c | c}
Peg & $1$ & $2$ & $3$ 
\\
\hline
\hline
$Q_1$ & $1$ & $2$ & $3$ 
\\
\hline
$Q_2$ & $1$ & $3$ & $4$ 
\\
\hline
$Q_3$ & $3$ & $2$ & $4$ 
\\
\hline
$Q_4$ & $2$ & $4$ & $1$ 
\\
\end{tabular}
\subcaption{$c=4$, $k=4$.}
\label{tab3a}
\end{subtable}
\hspace*{0.5em}
\begin{subtable}[b]{9em}
\begin{tabular}{c || c | c | c}
Peg & $1$ & $2$ & $3$ 
\\
\hline
\hline
$Q_1$ & $1$ & $3$ & $4$ 
\\
\hline
$Q_2$ & $2$ & $3$ & $4$ 
\\
\hline
$Q_3$ & $3$ & $1$ & $5$ 
\\
\hline
$Q_4$ & $4$ & $2$ & $5$ 
\\
\hline
$Q_5$ & $3$ & $5$ & $1$ 
\\
\hline
$Q_6$ & $4$ & $5$ & $3$ 
\\
\end{tabular}
\subcaption{$c=5$, $k=6$.}
\label{tab3b}
\end{subtable}
\hspace*{0.5em}
\begin{subtable}[b]{9em}
\begin{tabular}{c || c | c | c}
Peg & $1$ & $2$ & $3$ 
\\
\hline
\hline
$Q_1$ & $1$ & $2$ & $3$ 
\\
\hline
$Q_2$ & $1$ & $3$ & $2$ 
\\
\hline
$Q_3$ & $2$ & $1$ & $3$ 
\\
\hline
$Q_4$ & $2$ & $4$ & $1$ 
\\
\hline
$Q_5$ & $3$ & $5$ & $2$ 
\\
\hline
$Q_6$ & $5$ & $4$ & $6$ 
\\
\hline
$Q_7$ & $6$ & $5$ & $4$ 
\\
\end{tabular}
\subcaption{$c=6$, $k=7$.}
\label{tab3c}
\end{subtable}
\vspace*{1.1em}
\begin{subtable}[b]{9em}
\begin{tabular}{c || c | c | c}
Peg & $1$ & $2$ & $3$ 
\\
\hline
\hline
$Q_1$ & $1$ & $2$ & $7$ 
\\
\hline
$Q_2$ & $4$ & $1$ & $7$ 
\\
\hline
$Q_3$ & $2$ & $7$ & $5$ 
\\
\hline
$Q_4$ & $5$ & $7$ & $4$ 
\\
\hline
$Q_5$ & $7$ & $3$ & $2$ 
\\
\hline
$Q_6$ & $7$ & $4$ & $3$ 
\\
\hline
$Q_7$ & $6$ & $5$ & $1$ 
\\
\hline
$Q_8$ & $3$ & $6$ & $1$ 
\\
\hline
$Q_9$ & $3$ & $5$ & $6$ 
\\
\end{tabular}
\subcaption{$c=7$, $k=9$.}
\label{tab3d}
\end{subtable}
\hspace*{0.5em}
\begin{subtable}[b]{9em}
\begin{tabular}{c || c | c | c}
Peg & $1$ & $2$ & $3$ 
\\
\hline
\hline
$Q_1$ & $6$ & $5$ & $4$ 
\\
\hline
$Q_2$ & $3$ & $1$ & $5$ 
\\
\hline
$Q_3$ & $7$ & $6$ & $4$ 
\\
\hline
$Q_4$ & $8$ & $2$ & $6$ 
\\
\hline
$Q_5$ & $2$ & $4$ & $6$ 
\\
\hline
$Q_6$ & $2$ & $7$ & $5$ 
\\
\hline
$Q_7$ & $4$ & $1$ & $3$ 
\\
\hline
$Q_8$ & $8$ & $5$ & $2$ 
\\
\hline
$Q_9$ & $1$ & $6$ & $7$ 
\\
\hline
$Q_{10}$ & $4$ & $3$ & $8$ 
\\
\end{tabular}
\subcaption{$c=8$, $k=10$.}
\label{tab3e}
\end{subtable}
\hspace*{0.5em}
\begin{subtable}[b]{9em}
\begin{tabular}{c || c | c | c}
Peg & $1$ & $2$ & $3$ 
\\
\hline
\hline
$Q_1$ & $3$ & $1$ & $4$ 
\\
\hline
$Q_2$ & $2$ & $1$ & $3$ 
\\
\hline
$Q_3$ & $4$ & $2$ & $3$ 
\\
\hline
$Q_4$ & $1$ & $2$ & $4$ 
\\
\hline
$Q_5$ & $5$ & $7$ & $8$ 
\\
\hline
$Q_6$ & $5$ & $6$ & $7$ 
\\
\hline
$Q_7$ & $6$ & $8$ & $7$ 
\\
\hline
$Q_8$ & $7$ & $5$ & $8$ 
\\
\hline
$Q_9$ & $7$ & $3$ & $1$ 
\\
\hline
$Q_{10}$ & $7$ & $3$ & $5$ 
\\
\hline
$Q_{11}$ & $8$ & $9$ & $2$ 
\\
\hline
$Q_{12}$ & $8$ & $4$ & $9$ 
\\
\end{tabular}
\subcaption{$c=9$, $k=12$.}
\label{tab3f}
\end{subtable}
\caption{Feasible and optimal $(\lfloor (3c-1)/2 \rfloor)$-strategies 
for $ p=3$ and $ 4 \le c \le 9 $.}
\label{tab3}
\end{table}

Then we create 
$ \lfloor (3c-1)/2 \rfloor $-strategies for all $ c \ge 4$ in a way
similar to the case $p=2$.
We start with one of the 
six strategies of Table~\ref{tab3} and append the strategy of
Table~\ref{tab36} a suitable number of times,
shifting the colors of each such block of questions appropriately.

\begin{table}
\setlength{\tabcolsep}{2mm}
\centering
\begin{tabular}{c || c | c | c}
Peg & $1$ & $2$ & $3$ 
\\
\hline
\hline
$Q_1$ & $1$ & $5$ & $6$ 
\\
\hline
$Q_2$ & $4$ & $1$ & $6$ 
\\
\hline
$Q_3$ & $4$ & $5$ & $1$ 
\\
\hline
\hline
$Q_4$ & $2$ & $6$ & $4$ 
\\
\hline
$Q_5$ & $5$ & $2$ & $4$ 
\\
\hline
$Q_6$ & $5$ & $6$ & $2$ 
\\
\hline
\hline
$Q_7$ & $3$ & $4$ & $5$ 
\\
\hline
$Q_8$ & $6$ & $3$ & $5$ 
\\
\hline
$Q_9$ & $6$ & $4$ & $3$ 
\end{tabular}
\caption{Feasible (but not optimal)
strategy for $ p=3$,  $c = 6 $, $k=9$.}
\label{tab36}
\end{table}

Concretely, let 
$ c = 6s + t $, where $ s \in \NN_0 $, 
$ t \in \{4,5,6,7,8,9\}$ are uniquely determined.
For $ t=4,5,\dots,9 $ we start with the questions of the 
strategies in Tables~\ref{tab3a}--\ref{tab3f}, respectively.
Again we call this first group of 
questions the \emph{base questions}.

In all six cases the base questions
are followed $s$ 
times by the $9$ questions of Table~\ref{tab36}. 
For $l=1,2,\dots,s$, the $l$-th of these $s$ blocks is obtained by
adding the number $t+6(l-1)$
to the colors in Table~\ref{tab36}.
Again we call this second group of 
questions the \emph{iterated questions}.

As examples, we present the corresponding
$ (\lfloor (3c-1)/2 \rfloor)  $-strategies for 
$10 \le c \le 15$ explicitly in Table~\ref{tab31015}.

\begin{table}
\def\arraystretch{.97}
\setlength{\tabcolsep}{2mm}
\centering
\begin{subtable}[b]{9.5em}
\begin{tabular}{r || r | r | r}
Peg & $1$ & $2$ & $3$ 
\\
\hline
\hline
$Q_1$ & $1$ & $2$ & $3$ 
\\
\hline
$Q_2$ & $1$ & $3$ & $4$ 
\\
\hline
$Q_3$ & $3$ & $2$ & $4$ 
\\
\hline
$Q_4$ & $2$ & $4$ & $1$ 
\\
\hline
\hline
\hline
$Q_5$ & $5$ & $9$ & $10$ 
\\
\hline
$Q_6$ & $8$ & $5$ & $10$ 
\\
\hline
$Q_7$ & $8$ & $9$ & $5$ 
\\
\hline
\hline
$Q_8$ & $6$ & $10$ & $8$ 
\\
\hline
$Q_9$ & $9$ & $6$ & $8$ 
\\
\hline
$Q_{10}$ & $9$ & $10$ & $6$ 
\\
\hline
\hline
$Q_{11}$ & $7$ & $8$ & $9$ 
\\
\hline
$Q_{12}$ & $10$ & $7$ & $9$ 
\\
\hline
$Q_{13}$ & $10$ & $8$ & $7$ 
\end{tabular}
\subcaption{$c=10$, $k=13$.}
\label{tabb3a}
\end{subtable}
\hspace*{0.5em}
\begin{subtable}[b]{9.5em}
\begin{tabular}{r || r | r | r}
Peg & $1$ & $2$ & $3$ 
\\
\hline
\hline
$Q_1$ & $1$ & $3$ & $4$ 
\\
\hline
$Q_2$ & $2$ & $3$ & $4$ 
\\
\hline
$Q_3$ & $3$ & $1$ & $5$ 
\\
\hline
$Q_4$ & $4$ & $2$ & $5$ 
\\
\hline
$Q_5$ & $3$ & $5$ & $1$ 
\\
\hline
$Q_6$ & $4$ & $5$ & $3$ 
\\
\hline
\hline
\hline
$Q_7$ & $6$ & $10$ & $11$ 
\\
\hline
$Q_8$ & $9$ & $6$ & $11$ 
\\
\hline
$Q_9$ & $9$ & $10$ & $6$ 
\\
\hline
\hline
$Q_{10}$ & $7$ & $11$ & $9$ 
\\
\hline
$Q_{11}$ & $10$ & $7$ & $9$ 
\\
\hline
$Q_{12}$ & $10$ & $11$ & $7$ 
\\
\hline
\hline
$Q_{13}$ & $8$ & $9$ & $10$ 
\\
\hline
$Q_{14}$ & $11$ & $8$ & $10$ 
\\
\hline
$Q_{15}$ & $11$ & $9$ & $8$ 
\end{tabular}
\subcaption{$c=11$, $k=15$.}
\label{tabb3b}
\end{subtable}
\hspace*{0.5em}
\begin{subtable}[b]{9.5em}
\begin{tabular}{r || r | r | r}
Peg & $1$ & $2$ & $3$ 
\\
\hline
\hline
$Q_1$ & $1$ & $2$ & $3$ 
\\
\hline
$Q_2$ & $1$ & $3$ & $2$ 
\\
\hline
$Q_3$ & $2$ & $1$ & $3$ 
\\
\hline
$Q_4$ & $2$ & $4$ & $1$ 
\\
\hline
$Q_5$ & $3$ & $5$ & $2$ 
\\
\hline
$Q_6$ & $5$ & $4$ & $6$ 
\\
\hline
$Q_7$ & $6$ & $5$ & $4$ 
\\
\hline
\hline
\hline
$Q_8$ & $7$ & $11$ & $12$ 
\\
\hline
$Q_9$ & $10$ & $7$ & $12$ 
\\
\hline
$Q_{10}$ & $10$ & $11$ & $7$ 
\\
\hline
\hline
$Q_{11}$ & $8$ & $12$ & $10$ 
\\
\hline
$Q_{12}$ & $11$ & $8$ & $10$ 
\\
\hline
$Q_{13}$ & $11$ & $12$ & $8$ 
\\
\hline
\hline
$Q_{14}$ & $9$ & $10$ & $11$ 
\\
\hline
$Q_{15}$ & $12$ & $9$ & $11$ 
\\
\hline
$Q_{16}$ & $12$ & $10$ & $9$ 
\end{tabular}
\subcaption{$c=12$, $k=16$.}
\label{tabb3c}
\end{subtable}

\vspace*{1.1em}
\begin{subtable}[b]{9.5em}
\begin{tabular}{r || r | r | r}
Peg & $1$ & $2$ & $3$ 
\\
\hline
\hline
$Q_1$ & $1$ & $2$ & $7$ 
\\
\hline
$Q_2$ & $4$ & $1$ & $7$ 
\\
\hline
$Q_3$ & $2$ & $7$ & $5$ 
\\
\hline
$Q_4$ & $5$ & $7$ & $4$ 
\\
\hline
$Q_5$ & $7$ & $3$ & $2$ 
\\
\hline
$Q_6$ & $7$ & $4$ & $3$ 
\\
\hline
$Q_7$ & $6$ & $5$ & $1$ 
\\
\hline
$Q_8$ & $3$ & $6$ & $1$ 
\\
\hline
$Q_9$ & $3$ & $5$ & $6$ 
\\
\hline
\hline
\hline
$Q_{10}$ & $8$ & $12$ & $13$ 
\\
\hline
$Q_{11}$ & $11$ & $8$ & $13$ 
\\
\hline
$Q_{12}$ & $11$ & $12$ & $8$ 
\\
\hline
\hline
$Q_{13}$ & $9$ & $13$ & $11$ 
\\
\hline
$Q_{14}$ & $12$ & $9$ & $11$ 
\\
\hline
$Q_{15}$ & $12$ & $13$ & $9$ 
\\
\hline
\hline
$Q_{16}$ & $10$ & $11$ & $12$ 
\\
\hline
$Q_{17}$ & $13$ & $10$ & $12$ 
\\
\hline
$Q_{18}$ & $13$ & $11$ & $10$ 
\end{tabular}
\subcaption{$c=13$, $k=18$.}
\label{tabb3d}
\end{subtable}
\hspace*{0.5em}
\begin{subtable}[b]{9.5em}
\begin{tabular}{r || r | r | r}
Peg & $1$ & $2$ & $3$ 
\\
\hline
\hline
$Q_1$ & $6$ & $5$ & $4$ 
\\
\hline
$Q_2$ & $3$ & $1$ & $5$ 
\\
\hline
$Q_3$ & $7$ & $6$ & $4$ 
\\
\hline
$Q_4$ & $8$ & $2$ & $6$ 
\\
\hline
$Q_5$ & $2$ & $4$ & $6$ 
\\
\hline
$Q_6$ & $2$ & $7$ & $5$ 
\\
\hline
$Q_7$ & $4$ & $1$ & $3$ 
\\
\hline
$Q_8$ & $8$ & $5$ & $2$ 
\\
\hline
$Q_9$ & $1$ & $6$ & $7$ 
\\
\hline
$Q_{10}$ & $4$ & $3$ & $8$ 
\\
\hline
\hline
\hline
$Q_{11}$ & $9$ & $13$ & $14$ 
\\
\hline
$Q_{12}$ & $12$ & $9$ & $14$ 
\\
\hline
$Q_{13}$ & $12$ & $13$ & $9$ 
\\
\hline
\hline
$Q_{14}$ & $10$ & $14$ & $12$ 
\\
\hline
$Q_{15}$ & $13$ & $10$ & $12$ 
\\
\hline
$Q_{16}$ & $13$ & $14$ & $10$ 
\\
\hline
\hline
$Q_{17}$ & $11$ & $12$ & $13$ 
\\
\hline
$Q_{18}$ & $14$ & $11$ & $13$ 
\\
\hline
$Q_{19}$ & $14$ & $12$ & $11$ 
\end{tabular}
\subcaption{$c=14$, $k=19$.}
\label{tabb3e}
\end{subtable}
\hspace*{0.5em}
\begin{subtable}[b]{9.5em}
\begin{tabular}{c || c | c | c}
Peg & $1$ & $2$ & $3$ 
\\
\hline
\hline
$Q_1$ & $3$ & $1$ & $4$ 
\\
\hline
$Q_2$ & $2$ & $1$ & $3$ 
\\
\hline
$Q_3$ & $4$ & $2$ & $3$ 
\\
\hline
$Q_4$ & $1$ & $2$ & $4$ 
\\
\hline
$Q_5$ & $5$ & $7$ & $8$ 
\\
\hline
$Q_6$ & $5$ & $6$ & $7$ 
\\
\hline
$Q_7$ & $6$ & $8$ & $7$ 
\\
\hline
$Q_8$ & $7$ & $5$ & $8$ 
\\
\hline
$Q_9$ & $7$ & $3$ & $1$ 
\\
\hline
$Q_{10}$ & $7$ & $3$ & $5$ 
\\
\hline
$Q_{11}$ & $8$ & $9$ & $2$ 
\\
\hline
$Q_{12}$ & $8$ & $4$ & $9$ 
\\
\hline
\hline
\hline
$Q_{13}$ & $10$ & $14$ & $15$ 
\\
\hline
$Q_{14}$ & $13$ & $10$ & $15$ 
\\
\hline
$Q_{15}$ & $13$ & $14$ & $10$ 
\\
\hline
\hline
$Q_{16}$ & $11$ & $15$ & $13$ 
\\
\hline
$Q_{17}$ & $14$ & $11$ & $13$ 
\\
\hline
$Q_{18}$ & $14$ & $15$ & $11$ 
\\
\hline
\hline
$Q_{19}$ & $12$ & $13$ & $14$ 
\\
\hline
$Q_{20}$ & $15$ & $12$ & $14$ 
\\
\hline
$Q_{21}$ & $15$ & $13$ & $12$ 
\end{tabular}
\subcaption{$c=15$, $k=21$.}
\label{tabb3f}
\end{subtable}
\caption{Feasible and optimal 
$(\lfloor (3c-1)/2 \rfloor)$-strategies 
for $ p=3$ and $ 10 \le c \le 15 $.}
\label{tab31015}
\end{table}

In Subsections~\ref{ssec-three-pegs-feas}
and~\ref{ssec-three-pegs-opt} we prove
the following theorem.

\begin{theorem}
\label{thp3}
The presented strategy 
is a \emph{feasible} and \emph{optimal} 
$ \lfloor (3c-1)/2 \rfloor $-strategy 
for $p=3$ and for the corresponding $c \ge 4 $.
\end{theorem}

\begin{remark}
\begin{enumerate}[label=(\alph*)]
\item For even $c$,
our presented strategy 
for the Static Black-Peg AB~Game 
needs \emph{two questions less} than the strategy
from~\cite{JD20} for Static Black-Peg Mastermind
(see Strategies 1 and 3 in~\cite{JD20}).
\item For odd $c$,
our presented strategy for the Static Black-Peg 
AB~Game needs \emph{one question less} 
than the strategy from~\cite{JD20} 
for Static Black-Peg Mastermind
(see Strategies~2 and~4 in~\cite{JD20}).
\end{enumerate}
\end{remark}

\subsection{Feasibility of the
$ \lfloor (3c-1)/2 \rfloor $-Strategy}
\label{ssec-three-pegs-feas}

We start with the following observations.

\begin{observation}
\label {obs2}
For the $10$-strategy for $c=6$ of Table~\ref{tab36}
the following statements hold:
\begin{enumerate}[label=(\alph*)]
\item 
It consists entirely of $ (1,2,2) $-questions,
$ (2,1,2) $-questions and $ (2,2,1)$-questions.
\item It contains neighboring pairs of questions,
but no double neighboring pairs.
\item\label{obs2c} No color is missing on any of the pegs.
\item It consists of three blocks of three questions each,
all having the same structure,
namely one $ (1,2,2) $-question,
one $ (2,1,2) $-question 
and one $ (2,2,1)$-question.
Each two of them are neighboring,
but not double neighboring, and they are not
neighboring to any other question.
\end{enumerate}
\end{observation}

\begin{observation}
\label{obs3}
On each peg of the six $ \lfloor (3c-1)/2 \rfloor $-strategies 
of Table~\ref{tab3}, at most one color is missing
\end{observation}

Now we start with the proof of feasibility.
Consider one fixed $ (1,2,2) $-question,
$ (2,1,2)$-question or $ (2,2,1) $-question
which lies in one of the copies of 
the $9$-strategy of Table~\ref{tab36}. We begin
our reasoning by discussing which conclusions can
be drawn if the question receives a non-empty answer
$1$B, $2$B or~$3$B. 

\begin{enumerate}[label=(\Roman*)]

\item The answer is $1$B.

Here it is not clear which peg is correct. 
Again, this can be decided by the answers
to the two neighboring questions.
If \emph{both} answers to the neighboring questions 
are empty, then the color of the
non-overlapping peg is the correct one.
Otherwise the color of the 
peg which overlaps with the neighboring question 
whose answer contains the larger number of blacks 
is the correct one.

As an example, consider the strategy
for $c=12$ in Table~\ref{tabb3c}. 
For the secret $ (7 \e 9 \e 2) $, the answer
to the question $ Q_8 = (7 \e 11 \e 12) $ is $1$B.
The two neighboring questions to $Q_8$ are 
$ Q_9 = (10 \e 7 \e 12) $ and 
$ Q_{10} = (10 \e 11 \e 7) $.
As the answers to 
$ Q_9 $ and $ Q_{10} $ are both empty,
we know that color $7$ is correct on the first peg. 

On the other hand, for the secret 
$ (5 \e 11 \e 9) $, 
the answer to both the questions $ Q_8 $ and $Q_{10} $ 
is~$1$B and the answer to $ Q_9 $ is empty. 
Consequently, color $11$ must be correct on the second
peg.

\item The answer is $2$B.

Here it is not clear which two pegs are correct. 
However, this can be decided by the answers
to the two neighboring questions (which lie in the same
block) because \emph{at least} one of the answers to these
neighboring questions is also non-empty.
If \emph{both} answers 
are non-empty, then the colors of both overlapping 
pegs are correct. Otherwise the color of the 
peg which overlaps with the neighboring question 
that received the empty answer is the \emph{incorrect} one.

As an example, consider again the strategy
for $c=12$ in Table~\ref{tabb3c}. 
For the secret $ (7 \e 11 \e 3) $, the answer
to the question $Q_8$ is~$2$B.
As $Q_9$ gives the empty answer and 
$Q_{10}$ gives the answer~$1$B, we know that color~$12$ is
\emph{not} correct on the third peg, so that
color~$7$ is correct on the first peg and
color~$11$ is correct on the second peg. 

On the other hand, for the secret 
$ (8 \e 11 \e 12) $, 
the answer to the question $ Q_8 $ is $2$B,
and the answers to $ Q_9 $, $Q_{10} $ are $1$B. 
Then we know that color $11$ is correct on 
the second peg  
and that color $12$ is correct on the third peg.

\item The answer is $3$B.

The secret is found.
\end{enumerate}

So as soon we have a non-empty answer $i$B to a 
$ (1,2,2)$-question, a $ (2,1,2)$-question
or a $ (2,2,1) $-question
from one of the copies of 
the $9$-strategy of Table~\ref{tab36}, we can determine 
the $i$~pegs which gave rise to the non-empty answers,
and thus we know their colors. 

So assume that we start with the iterated questions
\emph{before} we ask the base questions. 
Then we have four cases after having seen the resulting
answers:
\begin{enumerate}[label=(\Roman*)]
\item No colors of any of the three pegs have been determined.

By Observation~\ref{obs2}\ref{obs2c}, this means that the correct
colors are those occurring in the base questions. Thus,
they are uniquely determined by the answers to those questions,
thanks to 
the feasibility of the strategies for $c=4,5,6,7,8,9$
in Table~\ref{tab3}. Their feasibility 
has been checked by a computer program using brute-force search~\cite{SC22}·

\item One peg has been determined.

Then the colors of two pegs are unknown. To prove that they
can be determined from the answers to the base questions,
it has to be shown that the 
$\lfloor (3c-1)/2 \rfloor $-strategies 
for $c=4,5,6,7,8,9$ in Table~\ref{tab3} remain feasible
if we remove one arbitrary 
column (corresponding to the already found peg).
The feasibility of these three sub-strategies
was again checked by the computer program~\cite{SC22}·

For motivation, in Table~\ref{ex3} 
we present an example of a feasible
$\lfloor (3c-1)/2 \rfloor $-strategy for $c=4$,
where adding one copy of the iterated questions
leads to an infeasible 
$\lfloor (3c-1)/2 \rfloor $-strategy for $c=10$.
For the latter, the two possible secrets
$ (1 \e 4 \e 5) $ and $ (2 \e 3 \e 5) $ lead to the same
combination of answers, namely
$1$B, $1$B, $1$B, $0$B, $0$B, $0$B, $1$B, $0$B, $\dots$, $0$B
which shows that the strategy is not feasible.
The reason is that the sub-strategy consisting of the
first two columns for $c=4$ is not feasible,
as the secrets $ (3 \e 1 ) $ and $ ( 4 \e 2 ) $
are indistinguishable.

\pagebreak[3]

\item Two pegs have been determined.

Only the color of one peg is still unknown. 
This peg is determined by the answers to the base questions 
as by Observation~\ref{obs3}) on each peg only one color is missing.

\item Three pegs have been determined.

The secret is found without making use of the base questions.

\end{enumerate}

So we have shown the feasibility 
of the $ \lfloor (3c-1)/2 \rfloor $-strategy 
in all cases.

\begin{table}
\setlength{\tabcolsep}{2mm}
\centering
\begin{subtable}[b]{9.5em}
\begin{tabular}{c || c | c | c}
Peg & $1$ & $2$ & $3$ 
\\
\hline
\hline
$Q_1$ & $1$ & $3$ & $2$ 
\\
\hline
$Q_2$ & $1$ & $3$ & $4$ 
\\
\hline
$Q_3$ & $2$ & $4$ & $3$ 
\\
\hline
$Q_4$ & $4$ & $1$ & $3$ 
\\
\end{tabular}
\subcaption{$c=4$, $k=4$.}
\label{ex3a}
\end{subtable}
\hspace*{0.5em}
\begin{subtable}[b]{9.5em}
\begin{tabular}{c || r | r | r}
Peg & $1$ & $2$ & $3$ 
\\
\hline
\hline
$Q_1$ & $1$ & $3$ & $2$ 
\\
\hline
$Q_2$ & $1$ & $3$ & $4$ 
\\
\hline
$Q_3$ & $2$ & $4$ & $3$ 
\\
\hline
$Q_4$ & $4$ & $1$ & $3$ 
\\
\hline
\hline
\hline
$Q_5$ & $5$ & $9$ & $10$ 
\\
\hline
$Q_6$ & $8$ & $5$ & $10$ 
\\
\hline
$Q_7$ & $8$ & $9$ & $5$ 
\\
\hline
\hline
$Q_8$ & $6$ & $10$ & $8$ 
\\
\hline
$Q_9$ & $9$ & $6$ & $8$ 
\\
\hline
$Q_{10}$ & $9$ & $10$ & $6$ 
\\
\hline
\hline
$Q_{11}$ & $7$ & $8$ & $9$ 
\\
\hline
$Q_{12}$ & $10$ & $7$ & $9$ 
\\
\hline
$Q_{13}$ & $10$ & $8$ & $7$ 
\end{tabular}
\subcaption{$c=10$, $k=13$.}
\label{ex3b}
\end{subtable}
\caption{Example strategy 
for $ p=3$ and $ c=4 $ and $c=10$, respectively.}
\label{ex3}
\end{table}

\subsection{Optimality of the
$ \lfloor (3c-1)/2 \rfloor $-Strategy}
\label{ssec-three-pegs-opt}

We use the following lemma, which in 
parts~\ref{item2a}-\ref{item2e}
is a generalization of Lemma~\ref{lem1}, 
and which holds only for $ c \ge 5 $.

\begin{lemma}
\label{lem2}
For Static Black-Peg AB~Game with $p=3$ and $ c \ge 5 $ the following statements hold:

\begin{enumerate}[label=(\alph*)]
\item
\label{item2a}
For each feasible strategy and for each peg there exists at most
one color which does not occur on this peg. 
\item
\label{item2b}
A feasible strategy cannot contain two
$ (1,1,\star) $-questions which are disjoint in the first two 
pegs.~\footnote{Recall that $(q_1\e q_2\e q_3)$ 
and $(q_1'\e q_2'\e q_3')$ are said to be disjoint in the first two pegs 
if $\{q_1,q_2\}\cap\{q_1',q_2'\}=\emptyset$.}

\item 
\label{item2c}
If a feasible strategy contains 
three $(1,1,\star)$-questions, 
by possibly permuting the colors, 
these questions can be written in the form 
$ (1 \e 2 \e \star_1) $, $ (2 \e 3 \e \star_2) $, 
$ (3 \e 1 \e \star_3) $,
for suitable colors $ \star_1 $, $\star_2$, $\star_3$.

\item
\label{item2d}
If a feasible strategy contains 
three $(1,1,\star)$-questions, 
then on at least one of the first two pegs of the  
questions, all $c$ colors must occur.

\item
\label{item2e}
A feasible strategy cannot contain 
four $(1,1,\star)$-questions. 

\item
\label{item2f}
Let a feasible strategy have a $ (1,1,\star) $-question
$ (q_1 \e q_2 \e \star_1) $, for a suitable color $ \star_1$, 
and let $q_3$ not occur on the first peg
and $q_4$ not occur on the second peg.
Then $ q_3\neq q_4 $ holds 
and, furthermore, $q_1=q_4$ or $q_2=q_3$.

\item
\label{item2g}
Let a feasible strategy have two 
$ (1,1,\star) $-questions,
and let $q_1$ not occur on the first peg
and $q_2$ not occur on the second peg.
Then, by permuting colors and reordering questions,
we can assume that $ q_1=3$ and $q_2 = 1 $, and the
two $ (1,1,\star) $-questions
can be written in the form
$ (1 \e 2 \e \star_1) $, $ (2 \e 3 \e \star_2) $,
for suitable colors $\star_1$, $\star_2$. 
\end{enumerate}

By symmetry, statements analogous to~\ref{item2b}-\ref{item2g} 
hold for the first and the third peg, and
for the second and the third peg.
\end{lemma}

\begin{proof}
\begin{enumerate}[font=\bfseries]
\item[\ref{item2a}]
Suppose that there are two colors $a$
and $b$ which do not occur on the first peg. 
Choose arbitrary
colors~$x$, $y$ with $x \neq y$ 
which are neither $a$ nor 
$b$, where $x$, $y$ exist 
because of $ c \ge 5$.  Then the possible secrets
$ (a \e x \e y) $ and $ (b \e x \e y) $ 
are indistinguishable,
contradicting the feasibility of the strategy.

\item[\ref{item2b}]
Suppose that the strategy contains two 
$(1,1,\star)$-questions.
By permuting colors and using the definition of a $ (1,1,\star)$-question, 
we can assume that these questions are $ (1 \e 2 \e \star_1) $ and 
$ (3 \e 4 \e \star_2) $, 
for suitable colors $ \star_1 $, $\star_2$.
Let $x$ be a color which is different from $1,2,3,4$,
where $x$ exists because of \mbox{$ c \ge 5$}.  
Then the possible secrets $ (1 \e 4 \e x ) $ 
and $ (3 \e 2 \e x) $ receive the same answers ($1$B or $2$B) 
for the first two questions and the same answer (0$B$ or 1$B$)
for all other questions, i.e., they are indistinguishable. 

\item[\ref{item2c}-\ref{item2e}]
The proofs work analogously to the proofs of
Lemma~\ref{lem1}\ref{item1c}-\ref{item1e}, 
respectively, if we add arbitrary entries on the third peg.

\item[\ref{item2f}] 

Let $x$ be an arbitrary
color not in $\{q_1,q_2,q_3,q_4\}$. Again, $x$ exists 
because of $c\ge5$.

First, suppose 
that $ q_3 = q_4 $.  By assumption,
$q_1\neq q_3=q_4\neq q_2$ holds,
and thus $ (q_1 \e q_3  \e x) $ and $ (q_3 \e q_2 \e x) $
are different and well-defined secrets that
lead to the same combination of answers. 
Thus, $ q_3 \neq q_4 $ holds.

Second, suppose that $ q_1 \neq q_4 $
and $ q_2 \neq q_3 $.
Then the possible secrets $ (q_1 \e q_4 \e x) $
and $ (q_3 \e q_2 \e x) $
lead to the same combination of answers. 
Thus, $ q_1 = q_4 $ or $ q_2 = q_3 $.

Note that for $c=4$ this statement does not hold.
(See the strategy of Table~\ref{ex3a}, where
the question $ Q_4 $ is 
a $ (1,1,\star$)-question and where the four numbers 
$q_1=4$, $q_2=1$, $q_3=3$, $q_4=2$ are disjoint.
Nevertheless the strategy is feasible.)

\item[\ref{item2g}] Using~\ref{item2b}, permuting colors, and
possibly switching
the two $ (1,1,\star)$-questions, we
can assume that these questions are 
$  (1 \e 2 \e \star_1) $ and $ (2 \e a \e \star_2) $, 
for suitable colors $\star_1$, $\star_2$. The
color $a$ cannot be equal to~$1$ because then $q_1,q_2 \notin \{1,2\} $ would hold 
and that would make the
possible secrets $(1\e q_2\e q_4)$ and $(q_1\e 2\e q_4)$
indistinguishable
(where $q_4$ is an arbitrary color with 
$q_4\notin \{1,2,q_1,q_2\}$).
Hence, we can assume that $a=3$, as claimed. 
Applying~\ref{item2f} to the questions 
$  (1 \e 2 \e \star_1) $ and $ (2 \e 3 \e \star_2) $, 
it follows that $q_1=2 \vee q_2=1$ \emph{and}
$q_1=3 \vee q_2=2$. As not both $q_1$ and $q_2$
can be equal to $2$, it follows that $ q_1=3, q_2=1$.
\hfill $ \square $
\end{enumerate}
\end{proof}

\begin{lemma}
\label{lem3}
For Static Black-Peg AB~Game with $p=3$ 
and $ c \ge 5 $, the following statements hold:

\begin{enumerate}[label=(\alph*)]
\item
\label{item3a}
Let a feasible strategy contain two 
$ (1,1,1) $-questions. Then it 
contains no additional $ (\star,1,1) $-question, 
$ (1,\star,1) $-question or
$ (1,1,\star) $-question. 

\item
\label{item3b}
A feasible strategy 
contains at most two $ (1,1,1) $-questions.
\item
\label{item3c}
Let a feasible strategy contain one
$ (1,1,1) $-question and let there be one color for each 
peg which does not occur
on this peg. Then the strategy 
contains no further $ (\star,1,1) $-questions, 
$ (1,\star,1) $-questions or
$ (1,1,\star) $-questions. 

\item
\label{item3d}
A feasible strategy,
where for each peg there exists one color
which does not occur on this peg, 
contains at most one $ (1,1,1) $-question.
\end{enumerate}
\end{lemma}

\begin{proof}
\begin{enumerate}[font=\bfseries]
\item[\ref{item3a}]
Let $ Q_1$, $Q_2 $ be the  two $ (1,1,1) $-questions.

Suppose that $ Q_3 $ is
a $ (1,1,\star)$-question in the strategy.
By Lemma~\ref{lem2}\ref{item2c}, we can assume
that the three questions have the form
\begin{eqnarray}
\label{stat1}
\begin{array}{c || c | c | c}
\mbox{Peg} & 1 & 2 & 3
\\
\hline
\hline
Q_1 & 1 & 2 & \star_1
\\
\hline
Q_2 & 2 & 3 & \star_2
\\
\hline
Q_3 & 3 & 1 & \star_3
\end{array}
\hspace*{0.5em},
\end{eqnarray}
where $ \star_1$, $\star_2$, $\star_3$ are suitable colors
and $ \star_1 \neq \star_2 $ (since $Q_1$ and $Q_2$ are $(1,1,1)$-questions).

By applying Lemma~\ref{lem2}\ref{item2b} to $Q_1$ and $Q_2$ twice,
namely to the first and third peg, and to the second and third peg,
the form can be assumed to be
\begin{eqnarray}
\label{stat2}
\begin{array}{c || c | c | c}
\mbox{Peg} & 1 & 2 & 3
\\
\hline
\hline
Q_1 & 1 & 2 & 3
\\
\hline
Q_2 & 2 & 3 & 1
\\
\hline
Q_3 & 3 & 1 & \star_3
\end{array}
\hspace*{0.5em}.
\end{eqnarray}
Then the possible secrets 
$ (2 \e 1 \e 3) $ and $ (3 \e 2 \e 1) $ 
receive the same combination of answers, 
namely $1$B for $ Q_1$, $Q_2$, $Q_3$
and $0$B for the remaining questions. 
This is a contradiction, and so
the strategy contains no
additional $ (1,1,\star)$-question.

By symmetry,
the strategy contains no additional
$ (1,\star,1)$-question or
$ (\star,1,1)$-question either. 

\item[\ref{item3b}]
This follows directly from~\ref{item3a}.

\item[\ref{item3c}]
Let $ Q_1' $ be the $ (1,1,1)$-question 
and $Q_2'$ be the question consisting of the colors which do 
not occur on the first, second and third peg, respectively.

Suppose that $ Q_3' $ is
a $ (1,1,\star)$-question in the strategy.
By Lemma~\ref{lem2}\ref{item2g}, we can assume
that the three questions have the form~\eqref{stat1}.
By Lemma~\ref{lem2}\ref{item2f}, we get 
the form~\eqref{stat2}.

Then the possible secrets 
$ (2 \e 1 \e 3) $ and $ (3 \e 2 \e 1) $ 
receive the same combination of answers, 
namely $1$B for $ Q_1'$ and $Q_3'$,
and $0$B for the remaining questions. 
This is a contradiction, and so
the strategy contains no
additional $ (1,1,\star)$-question.

Symmetrically, it follows that
the strategy contains no 
$ (1,\star,1)$-question and no
$ (\star,1,1)$-question either. 

\item[\ref{item3d}]
This follows directly from~\ref{item3c}.
\hfill $ \square $
\end{enumerate}
\end{proof}

\begin{lemma}
\label{lem4}
For Static Black-Peg AB~Game with $p=3$ 
and $ c \ge 5 $, the following statements hold:

\begin{enumerate}[label=(\alph*)]
\item
\label{item4a}
If a feasible strategy contains a
$ (1,1,1) $-question, then the strategy contains in total 
at most three questions which are
$ (\grt 2,1,1) $-questions, 
$ (1,\grt 2,1) $-questions or
$ (1,1,\grt 2) $-questions. 
\item
\label{item4b}
If a feasible strategy is such that,
for each peg, there exists one color which does 
not occur on this peg,
then the strategy contains in total 
at most three questions which are
$ (\grt 2,1,1) $-questions, 
$ (1,\grt 2,1) $-questions or
$ (1,1,\grt 2) $-questions.
\end{enumerate}
\end{lemma}

\begin{proof}
\begin{enumerate}[font=\bfseries]
\item[\ref{item4a}]
Let $ Q_1 $ be a  $ (1,1,1) $-question of the strategy.
We prove two assertions, which together prove part~\ref{item4a}. 

\begin{enumerate}[label=(\Roman*)]
\item\label{item4aI}
The strategy cannot have two 
$ (1,1,\grt 2) $-questions and two 
$ (1,\grt 2,1) $-questions at the same time.

Suppose that $ Q_2 $, $Q_3$ are two 
$ (1,1,\grt 2)$-questions and $ Q_4 $, $Q_5$ are two 
$ (1,\grt 2,1)$-questions in the strategy.
We present two indistinguishable
possible secrets, thus contradicting the assumed feasibility of the strategy.

By the definition of the AB~Game, Definition~\ref{def1}\ref{def1d} and
Lemma~\ref{lem2}\ref{item2c} applied to the first and the second 
peg and applied to the first and the third peg, we can assume
that the five questions have the form
\begin{eqnarray*}
\begin{array}{c || c | c | c}
\mbox{Peg} & 1 & 2 & 3
\\
\hline
\hline
Q_1 & 1 & 2 & 4
\\
\hline
\hline
Q_2 & 2 & 3 & \star_1
\\
\hline
Q_3 & 3 & 1 & \star_2
\\
\hline
\hline
Q_4 & 4 & \star_3 & 5
\\
\hline
Q_5 & 5 & \star_4 & 1
\end{array}
\end{eqnarray*}
where $\star_1,\star_2\notin\{1,4,5\}$ and $\star_3,\star_4\notin\{1,2,3\}$
are colors which occur at least twice on their peg.

Then the possible secrets 
$ (3 \e 2 \e 1) $ and $ (5 \e 1 \e 4) $ 
receive the same combination of answers, 
namely $1$B for $ Q_1$, $Q_3$, $Q_5$
and $0$B for the remaining questions.

Symmetrically, it follows that
the strategy cannot have two 
$ (1,1,\grt 2) $-question and two
$ (\grt 2,1,1) $-questions at the same time,
or $ (1,\grt 2,1) $-question and two
$ (\grt 2,1,1) $-questions at the same time.

\item\label{item4aII}
The strategy cannot have one 
$ (1,\grt 2,1) $-question, one 
$ (\grt 2,1,1) $-question, and two 
$ (1,1,\grt 2) $-questions at the same time.

Suppose that $ Q_2 $, $Q_3$ are the two
$ (1,1,\grt 2)$-questions, $ Q_4 $ is the 
$ (1,\grt 2,1)$-question and $ Q_5 $ is the 
$ (\grt 2,1,1)$-question of the strategy.
In each case or sub-case we present two indistinguishable
possible secrets, 
thus contradicting the assumed feasibility of the strategy.

By the definition of the AB~Game, Definition~\ref{def1}\ref{def1d} and
Lemma~\ref{lem2}\ref{item2c} applied to the first and the second 
peg, we can assume
that the five questions have the form
\begin{eqnarray*}
\label{stat4}
\begin{array}{c || c | c | c}
\mbox{Peg} & 1 & 2 & 3
\\
\hline
\hline
Q_1 & 1 & 2 & b
\\
\hline
Q_2 & 2 & 3 & \star_3
\\
\hline
Q_3 & 3 & 1 & \star_4
\\
\hline
\hline
Q_4 & 4 & \star_2 & d
\\
\hline
\hline
Q_5 & \star_1 & a & e
\end{array}
\hspace*{0.5em},
\end{eqnarray*}
where $\star_1$, $\star_2$, $\star_3$, $\star_4$ are colors
which occur at least twice on their peg,
and $a$, $b$, $d$, $e$ are colors\footnote{Note 
that the parameter $c$
is reserved for the number of colors. 
So we use the parameters $a,b,d,e,\dots$ here.}
which occur only once on their peg.

By Lemma~\ref{lem2}\ref{item2b} applied to the first and third peg
of questions $Q_1$ and $Q_4$, and to the second and third peg
of questions $Q_1$ and $Q_5$, we have
\begin{eqnarray*}
b=4 & \vee & d=1,
\end{eqnarray*}
and 
\begin{eqnarray*}
e=2 & \vee & a=b.
\end{eqnarray*}
This leads to the following four cases:

\begin{enumerate}[label=(\roman*)]
\item $a=b=4$.

Here the five questions have the form
\begin{eqnarray*}
\begin{array}{c || c | c | c}
\mbox{Peg} & 1 & 2 & 3
\\
\hline
\hline
Q_1 & 1 & 2 & 4
\\
\hline
\hline
Q_2 & 2 & 3 & \star_3
\\
\hline
Q_3 & 3 & 1 & \star_4
\\
\hline
\hline
Q_4 & 4 & \star_2 & d
\\
\hline
\hline
Q_5 & \star_1 & 4 & e
\end{array}
\hspace*{0.5em}.
\end{eqnarray*}
We have three sub-cases:

\begin{enumerate}[label=(\arabic*)]
\item All combinations of $ (d,e) $ except $d=3$ and except $e=1$. 

Here the possible secrets 
$ (3 \e 4 \e d) $ and $ (4 \e 1 \e e) $ 
receive the same combination of answers, 
namely $1$B for $ Q_3$, $Q_4$, $Q_5$
and $0$B for the remaining questions. 
(Note that $d \neq 4 $ and $ e \neq 4 $ as $Q_4$ and
$Q_5$ are questions in the AB~Game.)

\item $d=3$. 

Here the five questions have the form
\begin{eqnarray*}
\begin{array}{c || c | c | c}
\mbox{Peg} & 1 & 2 & 3
\\
\hline
\hline
Q_1 & 1 & 2 & 4
\\
\hline
\hline
Q_2 & 2 & 3 & \star_3
\\
\hline
Q_3 & 3 & 1 & \star_4
\\
\hline
\hline
Q_4 & 4 & \star_2 & 3
\\
\hline
\hline
Q_5 & \star_1 & 4 & e
\end{array}
\hspace*{0.5em}.
\end{eqnarray*}
Then the possible secrets 
$ (2 \e 4 \e 3) $ and $ (4 \e 3 \e e) $ 
receive the same combination of answers, 
namely $1$B for $ Q_2$, $Q_4$, $Q_5$
and $0$B for the remaining questions. 
(Note that $e \neq 3,4 $, as 
$Q_5 $ is a $ (\grt 2,1,1)$-question
having the color $e$ on the last peg.)

\pagebreak[4]

\item $e=1$. 

Here the five questions have the form
\begin{eqnarray*}
\begin{array}{c || c | c | c}
\mbox{Peg} & 1 & 2 & 3
\\
\hline
\hline
Q_1 & 1 & 2 & 4
\\
\hline
\hline
Q_2 & 2 & 3 & \star_3
\\
\hline
Q_3 & 3 & 1 & \star_4
\\
\hline
\hline
Q_4 & 4 & \star_2 & d
\\
\hline
\hline
Q_5 & \star_1 & 4 & 1
\end{array}
\hspace*{0.5em}.
\end{eqnarray*}
Then the possible secrets 
$ (1 \e 4 \e d) $ and $ (4 \e 2 \e 1) $ 
receive the same combination of answers, 
namely $1$B for $ Q_1$, $Q_4$, $Q_5$
and $0$B for the remaining questions. 
(Note that $d \neq 1,4 $, analogously to the case $d=3$.)

\end{enumerate}

\item $a \neq b=4 \wedge e=2$. 

Here the five questions have the form
\begin{eqnarray*}
\begin{array}{c || c | c | c}
\mbox{Peg} & 1 & 2 & 3
\\
\hline
\hline
Q_1 & 1 & 2 & 4
\\
\hline
\hline
Q_2 & 2 & 3 & \star_3
\\
\hline
Q_3 & 3 & 1 & \star_4
\\
\hline
\hline
Q_4 & 4 & \star_2 & d
\\
\hline
\hline
Q_5 & \star_1 & a & 2
\end{array}
\hspace*{0.5em}.
\end{eqnarray*}
Then the possible secrets 
$ (1 \e 3 \e 2) $ and $ (2 \e a \e 4) $ 
receive the same combination of answers, 
namely $1$B for $ Q_1$, $Q_2$, $Q_5$
and $0$B for the remaining questions. 
(Note that $a \neq 2,4 $.)

\pagebreak[3]

\item $d=1 \wedge e=2$. 

Here the five questions have the form
\begin{eqnarray*}
\begin{array}{c || c | c | c}
\mbox{Peg} & 1 & 2 & 3
\\
\hline
\hline
Q_1 & 1 & 2 & b
\\
\hline
\hline
Q_2 & 2 & 3 & \star_1
\\
\hline
Q_3 & 3 & 1 & \star_2
\\
\hline
\hline
Q_4 & 4 & \star_3 & 1
\\
\hline
\hline
Q_5 & \star_4 & a & 2
\end{array}
\hspace*{0.5em}.
\end{eqnarray*}
Then the possible secrets 
$ (3 \e a \e 1) $ and $ (4 \e 1 \e 2) $ 
receive the same combination of answers, 
namely $1$B for $ Q_3$, $Q_4$, $Q_5$
and $0$B for the remaining questions. 
(Note that $a \neq 1,3 $.)

\item $a=b \neq4 \wedge d=1$. 

\nopagebreak

Here the five questions have the form
\begin{eqnarray*}
\begin{array}{c || c | c | c}
\mbox{Peg} & 1 & 2 & 3
\\
\hline
\hline
Q_1 & 1 & 2 & a
\\
\hline
Q_2 & 2 & 3 & \star_3
\\
\hline
Q_3 & 3 & 1 & \star_4
\\
\hline
\hline
Q_4 & 4 & \star_2 & 1
\\
\hline
\hline
Q_5 & \star_1 & a & e
\end{array}
\hspace*{0.5em}.
\end{eqnarray*}
Then the possible secrets 
$ (3 \e 2 \e 1) $ and $ (4 \e 1 \e a) $ 
receive the same combination of answers, 
namely $1$B for $ Q_1$, $Q_3$, $Q_4$
and $0$B for the remaining questions. 
(Note that $a \neq 1,4 $.)

\end{enumerate}

Symmetrically, it follows that
the strategy cannot have one 
$ (1,1,\allowbreak\grt 2) $-question, one 
$ (\grt 2,1,1) $-question and two 
$ (1, \grt 2,1) $-questions, or one $ (1,1,\grt 2) $-question, one 
$ (1,\grt 2,1) $-question and two 
$ (\grt 2,1,1) $-questions at the same time. 
\end{enumerate}

\item[\ref{item4b}]
Consider a feasible strategy such that there are colors $q_1$, $q_2$, and $q_3$ 
which do not occur on pegs~1, 2, and 3, respectively. Let $Q_1=(q_1\e q_2\e 
q_3)$. We aim to apply Lemma~\ref{lem4}\ref{item4a} by adding the question $Q_1$ 
to the given strategy. To be able to do so, we have to verify that it is a valid 
question, i.e., the three colors are distinct. For doing so, we can assume that 
the strategy contains four questions which are
$ (\grt 2,1,1) $-questions, $ (1,\grt 2,1) $-questions or $ (1,1,\grt 2) 
$-questions (because otherwise the conclusion of  Lemma~\ref{lem4}\ref{item4a} is fulfilled).

By Lemma~\ref{lem2}\ref{item2d} and possibly switching the questions, 
we can assume that there are two $ (1,1,\grt 2) $-questions $Q_2$ and $Q_3$, a $ 
(1,\grt 2,1) $-question $Q_4$, and a question $Q_5$ which is either also a $ (1,\grt 2,1) $-question or a $ (\grt 2,1,1) $-question. This yields two cases.

Consider first the case where both $Q_4$ and $Q_5$ are $ (1,\grt 2,1) $-questions. Applying 
Lemma~\ref{lem2}\ref{item2g}, applied once to $Q_2$, $Q_3$ and once to $Q_4$, 
$Q_5$, yields the form
\begin{eqnarray*}
\begin{array}{c || c | c | c}
\mbox{Peg} & 1 & 2 & 3
\\
\hline
\hline
Q_2 & q_2 & \star_1 & \star_2
\\
\hline
Q_3 & \star_1 & q_1 & \star_3
\\
\hline
\hline
Q_4 & q_3 & \star_5 & \star_4
\\
\hline
Q_5 & \star_4 & \star_6 & q_1
\end{array}
\end{eqnarray*}
where $q_1\neq q_2$ and $q_1\neq q_3$,
and $\star_2$, $\star_3$, $\star_5$, $\star_6$ are colors
which occur at least twice on their peg,
and $\star_1$, $\star_4$ are colors occurring only once on their peg.
However, we also know that $Q_2$ is a $ (1,1,\grt 2) 
$-question, which implies that $q_2\neq q_3$. Hence, all three colors are distinct.

If $Q_5$ is a $ (\grt 2,1,1) $-question, the reasoning is easier: applying 
Lemma~\ref{lem2}\ref{item2f} thrice, namely to $Q_2$, $Q_4$, and $Q_5$ proves immediately that $q_1\neq q_2$, $q_1\neq q_3$, and $q_2\neq q_3$, respectively. Hence, again, all three colors $q_1$, $q_2$, and $q_3$ are distinct.

Thus, in either case $Q_1$ is a valid question. It follows that the 
strategy obtained by adding question $Q_1$ to the already feasible strategy is 
feasible. In this extended strategy, $Q_1$ is a $(1,1,1)$-question. Hence, Lemma~\ref{lem4}\ref{item4a} applies, the conclusion being that the modified strategy (and thus the original one) 
contains in total 
at most three questions which are
$ (\grt 2,1,1) $-questions, 
$ (1,\grt 2,1) $-questions or
$ (1,1,\grt 2) $-questions. This completes the proof.
\hfill $ \square $
\end{enumerate}

\end{proof}

\begin{lemma}
\label{lem5}
For Static Black-Peg AB~Game with $p=3$ 
and $ c \ge 5 $, the following statements hold:

\begin{enumerate}[label=(\alph*)]
\item
\label{item5a}
If a feasible strategy contains no $ (1,1,1) $-question, then it contains in total 
at most six questions that are $ (1,1,\grt 2) $-questions, 
$ (1,\grt 2,1) $-questions or
$ (\grt 2,1,1) $-questions. 
\item
\label{item5b}
If a feasible strategy contains no $ (1,1,1) $-question,
and for each of two pegs there is one color which does 
not occur on this peg, then the strategy contains in total 
at most five questions that are $ (1,1,\grt 2) $-questions, 
$ (1,\grt 2,1) $-questions or
$ (\grt 2,1,1) $-questions. 
\end{enumerate}
\end{lemma}

\begin{proof}
\begin{enumerate}[font=\bfseries]
\item[\ref{item5a}]
We prove two assertions, which together prove part~\ref{item5a}. 

\begin{enumerate}[label=(\Roman*)]
\item The strategy cannot have three 
$ (1,1,\grt 2) $-questions, three $ (1,\grt 2,1) $-questions 
and one $ (\grt 2,1,1) $-question at the same time.

Suppose that $Q_1$, $ Q_2 $, $Q_3$ are 
$ (1,1,\grt 2)$-questions, $ Q_4 $, $Q_5$, $Q_6$ 
are $ (1,\grt 2,1)$-questions and
$ Q_7 $ is a $ (\grt 2,1,1) $-question of the strategy.

By Lemma~\ref{lem2}\ref{item2c} applied to the first and the second 
peg and applied to the first and the third peg, we can assume
that the seven questions have the form
\begin{eqnarray*}
\begin{array}{c || c | c | c}
\mbox{Peg} & 1 & 2 & 3
\\
\hline
\hline
Q_1 & 1 & 2 & \star_5
\\
\hline
Q_2 & 2 & 3 & \star_6
\\
\hline
Q_3 & 3 & 1 & \star_7
\\
\hline
\hline
Q_4 & 4 & \star_2 & 5
\\
\hline
Q_5 & 5 & \star_3 & 6
\\
\hline
Q_6 & 6 & \star_4 & 4
\\
\hline
\hline
Q_7 & \star_1 & a & b
\end{array}
\hspace*{0.5em},
\end{eqnarray*}
where $\star_1$, $\star_2$, $\star_3$, $\star_4$, $\star_5$,
$\star_6$, $\star_7$ are colors
which occur at least twice on their peg,
and $a$, $b$ are colors occurring only once on their peg.
In particular, $ a \notin \{1,2,3\} $ and 
$ b \notin \{4,5,6\} $.
Note that the appearance of a $7$-th
color~$\star_1$ in the questions already proves the assertion for
$c=5,6$ (or, expressed differently, that the assumptions imply $c\ge 7$).

Choose a question $Q = (q_1\e q_2\e q_3) $ 
from the set 
$ S_1 := \{Q_1,Q_2,Q_3\} $ so that $b \neq q_2 $.
(Note that at least two questions in $ S_1$ fulfill this condition.)

Similarly, choose a question $Q' = (q_4\e q_5\e q_6) $ 
from the the set 
$ S_2 := \{Q_4,Q_5,Q_6\} $ so that $a \neq q_6 $.

By construction, the possible secrets 
$ (q_1 \e a \e q_6) $ and $ (q_4 \e q_2 \e b) $ 
receive the same combination of answers, 
namely $1$B for $Q$, $Q'$, $Q_7$
and $0$B for the remaining questions, contradicting
the feasibility of the strategy. 

As an example for $ a=4$, $b=3$ choose $ Q = Q_3 $ and $ Q' = Q_5 $. 
Then the possible secrets 
$ (3 \e 4 \e 6) $ and $ (5 \e 1 \e 3) $ 
receive the same combination of answers, 
namely $1$B for $Q_3$, $Q_5$, $Q_7$
and $0$B for the remaining questions. 

Symmetrically, it follows that 
the strategy cannot have three 
$ (1,1,\allowbreak\grt 2) $-questions, three $ (\grt 2,1,1) $-questions 
and one $ (1,\grt 2,1) $-question, or three 
$ (1,\grt 2,1) $-questions, three $ (\grt 2,1,1) $-questions 
and one $ (1, 1,\allowbreak\grt 2) $-question at the same time.

\item The strategy cannot have three 
$ (1,1,\grt 2) $-questions, two $ (1,\grt 2,1) $-questions 
and two $ (\grt 2,1,1) $-questions at the same time.

Suppose that $Q_1$, $ Q_2 $ and $Q_3$ are 
$ (1,1,\grt 2)$-questions, $ Q_4 $ and $Q_5$ are 
$ (1,\grt 2,1)$-questions and $Q_6$ and 
$ Q_7 $ are $ (\grt 2,1,1) $-questions of the strategy.
Below, we consider an exhaustive set of cases, presenting in each
of them two indistinguishable possible secrets, thus contradicting
the feasibility of the strategy.

By Lemma~\ref{lem2}\ref{item2b}, 
applied to the first and the third 
peg and applied to the second and the third peg, 
by possibly switching questions $Q_4$ and $Q_5$,
and questions $Q_6$ and $Q_7$, and by applying
Lemma~\ref{lem2}\ref{item2c} to the first and the second peg, 
we can assume that the seven questions have the form
\begin{eqnarray*}
\begin{array}{c || c | c | c}
\mbox{Peg} & 1 & 2 & 3
\\
\hline
\hline
Q_1 & 1 & 2 & \star_5
\\
\hline
\hline
Q_2 & 2 & 3 & \star_6
\\
\hline
Q_3 & 3 & 1 & \star_7
\\
\hline
\hline
Q_4 & 4 & \star_3 & d
\\
\hline
Q_5 & 5 & \star_4 & 4
\\
\hline
\hline
Q_6 & \star_1 & a & e
\\
\hline
Q_7 & \star_2 & b & a
\end{array}
\hspace*{0.5em},
\end{eqnarray*}
where $\star_1$, $\star_2$, $\star_3$, $\star_4$, $\star_5$,
$\star_6$, $\star_7$ are colors
which occur at least twice on their peg.
and $a$, $b$, $d$, $e$ are colors occurring
only once on their peg.

Here the appearance of a $6$-th
color~$\star_1$ in the questions already proves the assertion for
$c=5$ (or, expressed differently, that the assumptions imply $c\ge 6$).

We consider the following four cases:

\begin{enumerate}[label=(\roman*)]
\item $a\neq 5 \wedge b \neq 4$.

\nopagebreak

Then the possible secrets 
$ (1 \e b \e 4) $ and $ (5 \e 2 \e a) $ 
receive the same combination of answers, 
namely $1$B for $ Q_1$, $Q_5$, $Q_7$
and $0$B for the remaining questions. 
(Note that $a\not= 2,5$ and $b \not=1,4 $.)

\item $a=5 \wedge e \neq 3$.

Here the seven questions have the form
\begin{eqnarray*}
\begin{array}{c || c | c | c}
\mbox{Peg} & 1 & 2 & 3
\\
\hline
\hline
Q_1 & 1 & 2 & \star_5
\\
\hline
Q_2 & 2 & 3 & \star_6
\\
\hline
Q_3 & 3 & 1 & \star_7
\\
\hline
\hline
Q_4 & 4 & \star_3 & d
\\
\hline
Q_5 & 5 & \star_4 & 4
\\
\hline
\hline
Q_6 & \star_1 & 5 & e
\\
\hline
Q_7 & \star_2 & b & 5
\end{array}
\hspace*{0.5em},
\end{eqnarray*}
Then the possible secrets 
$ (2 \e 5 \e 4) $ and $ (5 \e 3 \e e) $ 
receive the same combination of answers, 
namely $1$B for $ Q_2$, $Q_5$, $Q_6$
and $0$B for the remaining questions. 
(Note that $e\not= 3,5$.)

\pagebreak[3]

\item $a=5 \wedge e=3$.

Here the seven questions have the form
\begin{eqnarray*}
\begin{array}{c || c | c | c}
\mbox{Peg} & 1 & 2 & 3
\\
\hline
\hline
Q_1 & 1 & 2 & \star_5
\\
\hline
Q_2 & 2 & 3 & \star_6
\\
\hline
Q_3 & 3 & 1 & \star_7
\\
\hline
\hline
Q_4 & 4 & \star_3 & d
\\
\hline
Q_5 & 5 & \star_4 & 4
\\
\hline
\hline
Q_6 & \star_1 & 5 & 3
\\
\hline
Q_7 & \star_2 & b & 5
\end{array}
\hspace*{0.5em},
\end{eqnarray*}
Then the possible secrets 
$ (1 \e 5 \e 4) $ and $ (5 \e 2 \e 3) $ 
receive the same combination of answers, 
namely $1$B for $ Q_1$, $Q_5$, $Q_6$
and $0$B for the remaining questions. 

\pagebreak[4]

\item $a\neq5 \wedge b=4 \wedge d\neq1 $.

Here the seven questions have the form
\begin{eqnarray*}
\begin{array}{c || c | c | c}
\mbox{Peg} & 1 & 2 & 3
\\
\hline
\hline
Q_1 & 1 & 2 & \star_5
\\
\hline
Q_2 & 2 & 3 & \star_6
\\
\hline
Q_3 & 3 & 1 & \star_7
\\
\hline
\hline
Q_4 & 4 & \star_3 & d
\\
\hline
Q_5 & 5 & \star_4 & 4
\\
\hline
\hline
Q_6 & \star_1 & a & e
\\
\hline
Q_7 & \star_2 & 4 & a
\end{array}
\hspace*{0.5em},
\end{eqnarray*}
Then the possible secrets 
$ (1 \e 4 \e d) $ and $ (4 \e 2 \e a) $ 
receive the same combination of answers, 
namely $1$B for $ Q_1$, $Q_4$, $Q_7$
and $0$B for the remaining questions. 
(Note that $a\not= 2, 4$ and $ d \not= 1,4$.)

\pagebreak[3]

\item $a\neq5 \wedge b=4 \wedge d=1 $.

Here the seven questions have the form
\begin{eqnarray*}
\begin{array}{c || c | c | c}
\mbox{Peg} & 1 & 2 & 3
\\
\hline
\hline
Q_1 & 1 & 2 & \star_5
\\
\hline
Q_2 & 2 & 3 & \star_6
\\
\hline
Q_3 & 3 & 1 & \star_7
\\
\hline
\hline
Q_4 & 4 & \star_3 & 1
\\
\hline
Q_5 & 5 & \star_4 & 4
\\
\hline
\hline
Q_6 & \star_1 & a & e
\\
\hline
Q_7 & \star_2 & 4 & a
\end{array}
\hspace*{0.5em},
\end{eqnarray*}
Then the possible secrets 
$ (3 \e 4 \e 1) $ and $ (4 \e 1 \e a) $ 
receive the same combination of answers, 
namely $1$B for $ Q_3$, $Q_4$, $Q_7$
and $0$B for the remaining questions. 
(Note that $a\not= 1,4$.)
\end{enumerate}

Again, symmetric arguments prove that 
the strategy cannot simultaneously have three 
$ (1,\grt 2,1) $-questions, two $ (\grt 2,1,1) $-questions 
and two $ (1, 1,\allowbreak\grt 2) $-questions, or three 
$ (\grt 2,1,1) $-questions, two $ (1,\grt 2,1) $-questions 
and two $ (1, 1,\grt 2) $-questions.

This completes the proof of \ref{item5a}.
\end{enumerate}

\item[\ref{item5b}]
Similarly to the proof of Lemma~\ref{lem4}\ref{item4b}, this case can be proved 
by modifying the strategy in order to be able to apply 
Lemma~\ref{lem5}\ref{item5a}. Let $q_1,q_2$ be the colors which do 
not occur on the first and second peg, respectively. By 
Lemma~\ref{lem2}\ref{item2f}, $q_1\neq q_2$ holds. Now, let $ Q_1 = (q_1\e q_2\e q_3)$, 
where the third color $q_3$ is chosen in such a way that it 
appears at least once on the third peg of another question but is different from 
$q_1$ and $q_2$. Such a color $q_3$ exists because 
by Lemma~\ref{lem2}\ref{item2a} there are at least $c-1\ge 4$ 
distinct colors on the third peg, so we have at least two to choose from (as we 
have to exclude $q_1$ and $q_2$).

We have thus shown that we can add $Q_1$ to the given strategy, where it becomes 
a $(1,1,\grt 2)$-question. The extended strategy fulfills the assumptions of 
Lemma~\ref{lem5}\ref{item5a}, which completes the proof.
\hfill $ \square $
\end{enumerate}
\end{proof}

Now we come to the actual proof of optimality.
For $c=4$ it can be easily checked by brute-force search
that there is no $ (\lfloor (3c-1)/2 \rfloor-1) $-strategy,
i.e., no strategy with $ k = \lfloor (3 \cdot 4-1)/2  
\rfloor -2 = 3 $ questions.
Thus, assume in the following that $ c \ge 5 $.
The presented $ \lfloor (3c-1)/2 \rfloor $-strategy
has $k=\lfloor (3c-1)/2 \rfloor-1$ questions. We
show that this is optimal by considering any feasible strategy
and proving that it must contain at least $k$ questions. 

In the following let $e$ be the number of $ (1,1,1)$-questions 
and $f$ be the total number of 
$ (1,1,\grt 2)$-questions, $ (1,\grt 2,1)$-questions and
$ (\grt 2,1,1)$-questions of this assumed strategy. Further,
for $i=1,2,3$ let $l_i$ be the number of colors which
occur exactly once on the $i$-th peg of the strategy.

We will need the following observation, which is easy to see.

\begin{observation}
\label{lrel}
Every feasible strategy contains at least
$$
l_1 + l_2 + l_3 - 2e -f
$$
questions.
Hence, we can establish optimality of our strategy by showing
that $l_1 + l_2 + l_3 - 2e -f\ge k$.
\end{observation}

We distinguish between even and odd numbers of colors.

\begin{enumerate}[label=(\Roman*)]
\item $c$ even.
\label{heven}

We have
\begin{eqnarray*}
k & = & \frac{3c}{2} - 2.
\end{eqnarray*}
By Lemma~\ref{lem2}\ref{item2a} for each peg  
there exists at most one color which does not  
occur throughout the entire strategy.
We start by proving two claims which verify
optimality in two restricted cases.

\begin{claim}\label{claim missing color}
If $ l_i \le \frac{c}{2} $
for some peg $i \in \{1,2,3\} $ on which
not all colors occur, then the strategy contains
at least $k$ questions.
\end{claim}

This is clear because $ l_i \le \frac{c}{2} $ means that
$\frac c2-1$ colors appear at least twice on peg~$i$, yielding
at least
\begin{eqnarray*}
\frac{c}{2} + 2 \cdot 
\left( \frac{c}{2}-1 \right) \; = \; 
\frac{3c}{2}-2 \; = \; k
\end{eqnarray*}
questions.

\pagebreak[3]

\begin{claim}\label{claim no missing color}
If $ l_i \le \frac{c}{2}+2 $
for some peg $i \in \{1,2,3\} $ on which
all colors occur, then the strategy contains
at least $k$ questions.
\end{claim}

Here, the assumption $ l_i \le \frac{c}{2}+2 $
implies that we have at least 
\begin{eqnarray*}
\frac{c}{2} + 2 + 2 \cdot 
\left( \frac{c}{2}-2 \right) \; = \; 
\frac{3c}{2}-2 \; = \; k
\end{eqnarray*}
questions.

Now we finish the case~\ref{heven} of the proof by
considering two sub-cases. 

\begin{enumerate}[label=(\roman*)]
\item 
\label{casc}
For each of the three pegs there exists 
one color which does not occur throughout the entire
strategy.

By Claim~\ref{claim missing color}, the situation which remains
to be considered is when $l_i\ge\frac c2+1$ for every $i\in\{1,2,3\}$.
In this case,
\begin{eqnarray*}
l_1+l_2+l_3 \;\ge\; 3 \cdot \left( \frac{c}{2} + 1 \right) 
\; = \; \frac{3c}{2} +3 \; = \;
k + 5.
\end{eqnarray*}
By Observation~\ref{lrel}, it thus remains to be shown that
\begin{eqnarray*}
\label{rel5}
2e + f & \le & 5.
\end{eqnarray*}
By Lemma~\ref{lem3}~\ref{item3b}, at most 
two $ (1,1,1) $-questions exist.

\begin{enumerate}[label=(\arabic*)]
\item\label{caseAi} If we have two $(1,1,1)$-questions, then 
Lemma~\ref{lem3}~\ref{item3a} yields
\begin{eqnarray*}
2e +f \; = \; 2 \cdot 2 + 0 \; = \; 4. 
\end{eqnarray*}
\item\label{caseAii} If we have one $ (1,1,1)$-question, then 
Lemma~\ref{lem4}~\ref{item4a} yields
\begin{eqnarray*}
2e +f \; \le \; 2 \cdot 1 + 3 \; = \; 5. 
\end{eqnarray*}
\item If we have no $ (1,1,1)$-question, then 
Lemma~\ref{lem4}~\ref{item4b} yields
\begin{eqnarray*}
2e +f \; \le \; 2 \cdot 0 + 3 \; = \; 3. 
\end{eqnarray*}
\end{enumerate}

So we have finished the optimality proof for case~\ref{casc}.

\item For at most two of the three pegs there exists 
one color which does not occur throughout the entire
strategy.
\label{casd}

By Claims~\ref{claim missing color} and~\ref{claim no missing color},
it remains to consider the situation where $l_i\ge\frac c2+1$ for at most
two $i\in\{1,2,3\}$ and $l_i\ge\frac{c}{2}+3$ for the remaining ones, i.e.,
\begin{eqnarray*}
l_1+l_2+l_3&\ge&\left( \frac{c}{2} + 3 \right) 
+ 2 \cdot \left( \frac{c}{2} + 1 \right) \\
&=&\frac{3c}{2} +5 \\
&=&k + 7. 
\end{eqnarray*}
Again using Observation~\ref{lrel}, we thus have established that there are
at least $k$~questions if we can show that
\begin{eqnarray*}
\label{rel7}
2e + f & \le & 7.
\end{eqnarray*}
Again we have at most 
two $ (1,1,1) $-questions.

\begin{enumerate}[label=(\arabic*)]
\item If we have two $ (1,1,1)$-questions,
as in case~\ref{casc}\ref{caseAi}
\begin{eqnarray*}
2e +f \; = \; 2 \cdot 2 + 0 \; = \; 4. 
\end{eqnarray*}
\item If we have one $ (1,1,1)$-question, 
as in case~\ref{casc}\ref{caseAii}
\begin{eqnarray*}
2e +f \; \le \; 2 \cdot 1 + 3 \; = \; 5. 
\end{eqnarray*}
\item If we have no $ (1,1,1)$-question, then 
by Lemma~\ref{lem5}~\ref{item5a}
\begin{eqnarray*}
2e +f \; \le \; 2 \cdot 0 + 6 \; = \; 6. 
\end{eqnarray*}
\end{enumerate}

So we have finished the optimality proof for case~\ref{casd} and thus for case~\ref{heven}.
\end{enumerate}

\item $c$ odd.
\label{hodd}

We have
\begin{eqnarray*}
k & = & \frac{3c}{2} -\frac{3}{2}.
\end{eqnarray*}
We proceed similarly to case~\ref{casc}, starting by proving two claims.

\begin{claim}\label{claim missing color2}
If $ l_i \le \frac{c-1}{2} $
for some peg $i \in \{1,2,3\} $ on which
not all colors occur, then the strategy contains
at least $k$ questions.
\end{claim}

Clearly, if $ l_i \le \frac{c-1}{2} $,
then we have at least 
\begin{eqnarray*}
\frac{c-1}{2} + 2 \cdot 
\left( \frac{c-1}{2} \right) \; = \; 
\frac{3c}{2}-\frac{3}{2} \; = \; k
\end{eqnarray*}
questions, which verifies the claim.

\begin{claim}\label{claim no missing color2}
If $ l_i \le \frac{c+1}{2}+1 $
for some peg $i \in \{1,2,3\} $ on which
all colors occur, then the strategy contains
at least $k$ questions.
\end{claim}

This is clear as well, because the assumption $ l_i \le \frac{c+1}{2}+1 $,
together with the fact that all colors occur on peg~$i$, implies that
we have at least 
\begin{eqnarray*}
\frac{c+1}{2} + 1 + 2 \cdot 
\left( \frac{c-1}{2}-1 \right) \; \ge \; 
\frac{3c}{2}-\frac{3}{2} \; = \; k
\end{eqnarray*}
questions.

Now we finish the part~\ref{hodd} of the proof by
considering three sub-cases. 

\begin{enumerate}[label=(\roman*)]
\item\label{case}
For each of the three pegs there exists 
one color which does not occur throughout the entire
strategy.

By Claim~\ref{claim missing color2}, we only need to consider the situation
where $l_i\ge\frac{c-1}{2}+1$ for every $i\in\{1,2,3\}$,
which means that
\begin{eqnarray*}
l_1+l_2+l_3\;\ge\;3 \cdot \left( \frac{c+1}{2} \right) 
\; = \; \frac{3c}{2} +\frac{3}{2} \; = \;
k + 3.
\end{eqnarray*}
By Observation~\ref{lrel}, we thus have to show that 
\begin{eqnarray*}
\label{rel3}
2e + f & \le & 3.
\end{eqnarray*}
By Lemma~\ref{lem3}~\ref{item3d}, at most 
one $ (1,1,1) $-questions exists, which yields two cases:

\begin{enumerate}[label=(\arabic*)]
\item If we have one $ (1,1,1)$-question, then 
by Lemma~\ref{lem3}~\ref{item3c}
\begin{eqnarray*}
2e +f \; = \; 2 \cdot 1 + 0 \; = \; 2. 
\end{eqnarray*}
\item If we have no $ (1,1,1)$-question, then 
by Lemma~\ref{lem4}~\ref{item4b}
\begin{eqnarray*}
2e +f \; \le \; 2 \cdot 0 + 3 \; = \; 3. 
\end{eqnarray*}
\end{enumerate}

So we have completed case~\ref{case}.

\item\label{casf} For exactly two of the three pegs there exists 
one color which does not occur 
throughout the entire strategy.

By Claims~\ref{claim missing color2} and~\ref{claim no missing color2},
we need to consider the case where $l_i\ge\frac{c-1}{2}+1$ for exactly 
two pegs $i\in\{1,2,3\}$ and $ l_i \ge \frac{c+1}{2}+2 $ for the remaining
peg, which yields
\begin{eqnarray*}
l_1+l_2+l_2&\ge&\left( \frac{c+1}{2} + 2 \right) 
+ 2 \cdot \left( \frac{c+1}{2} \right) \\
&=&\frac{3c}{2} +\frac{7}{2} \\
&=& k + 5. 
\end{eqnarray*}
Again using Observation~\ref{lrel}, this means that we have to show that
\begin{eqnarray*}
\label{rel5b}
2e + f & \le & 5.
\end{eqnarray*}
By Lemma~\ref{lem3}~\ref{item3b}, at most 
two $ (1,1,1) $-questions exist, and so we have the following cases:

\begin{enumerate}[label=(\arabic*)]
\item If there are two $ (1,1,1)$-questions, then 
by Lemma~\ref{lem3}~\ref{item3a}
\begin{eqnarray*}
2e +f \; = \; 2 \cdot 2 + 0 \; = \; 4. 
\end{eqnarray*}
\item If there is one $ (1,1,1)$-question, then 
by Lemma~\ref{lem4}~\ref{item4a}
\begin{eqnarray*}
2e +f \; \le \; 2 \cdot 1 + 3 \; = \; 5. 
\end{eqnarray*}
\item If there is no $ (1,1,1)$-question, then 
by Lemma~\ref{lem5}\ref{item5b}
\begin{eqnarray*}
2e +f \; \le \; 2 \cdot 0 + 5 \; = \; 5. 
\end{eqnarray*}
\end{enumerate}

So we have finished the proof for case~\ref{casf}.

\item\label{casg} For at most one of the three pegs there exists 
a color which does not occur throughout the entire
strategy.

Again using Claims~\ref{claim missing color2} and~\ref{claim no missing color2},
the case to be studied is the one where
\begin{eqnarray*}
l_1+l_2+l_3&\ge& 2 \cdot \left( \frac{c+1}{2} + 2 \right) 
+ \left( \frac{c+1}{2} \right) \\
&=&\frac{3c}{2} +\frac{11}{2}\\
&=&k + 7. 
\end{eqnarray*}
By Observation~\ref{lrel}
we thus need to show that
\begin{eqnarray*}
\label{rel7b}
2e + f & \le & 7.
\end{eqnarray*}
Again we have at most 
two $ (1,1,1) $-questions, and thus the following cases:

\begin{enumerate}[label=(\arabic*)]
\item If there are two $ (1,1,1)$-questions, then 
by Lemma~\ref{lem3}~\ref{item3a}
\begin{eqnarray*}
2e +f \; = \; 2 \cdot 2+0 \; = \; 4. 
\end{eqnarray*}
\item If there is one $ (1,1,1)$-question, then 
by Lemma~\ref{lem4}~\ref{item4a}
\begin{eqnarray*}
2e +f \; \le \; 2 \cdot 1 + 3 \; = \; 5. 
\end{eqnarray*}
\item If there is no $ (1,1,1)$-question, then 
by Lemma~\ref{lem5}~\ref{item5a}
\begin{eqnarray*}
2e +f \; \le \; 2 \cdot 0 + 6 \; = \; 6. 
\end{eqnarray*}
\end{enumerate}

So the proof for case~\ref{casg}
is complete, and so is the optimality proof in its entirety.
\hfill $ \square $
\end{enumerate}
\end{enumerate}

\section{Conclusions and Future Work}
\label{sec-futurework}

We have introduced a new game, called Static Black-Peg AB~Game,
a variant of the AB~Game for which, to the best of our knowledge,
only the case of equal number of pegs and colors (called
``Static Permutation Mastermind'') has been studied before~\cite{GJSS21,LMS22}. 
We have developed optimal strategies for this game
for two and three pegs, extending similar results for Static 
Black-Peg Mastermind with two and three pegs~\cite{Jag16,JD18,JD20}.

Some of the methods and strategies used in this work would also simplify the 
earlier strategies~\cite{Jag16,JD20} of the related Static Black-Peg Mastermind, which makes them easier to read and reason about.
In particular, the iterated questions of Table~\ref{tab234c} and \ref{tab36}
can be used in optimal strategies for Static Black-Peg Mastermind
with two and three pegs, respectively. 

It would be an interesting task
to extend the strategies to more than three pegs.
However, even though such strategies obviously exist, 
the optimality proofs will likely be quite challenging.
An attempt to do so should probably aim at developing a
schema for constructing optimal strategies for an arbitrary
numbers of pegs $p$. The fact that our strategies for $p=2$ and
$p=3$ pegs clearly share some structural characteristics may
provide clues. A definition of optimal strategies parameterized
with $p$ may also facilitate an optimality proof that works for
all of them.

Further, solving Static Black-Peg Mastermind is equivalent to 
the well-studied problem of determining 
the metric dimension of undirected graphs, in our case the graph $ \ZZ_c^p $.
Thus, the fact that the strategy construction principle of iterated blocks used in this paper applies to Static Black-Peg Mastermind as well yields simplified proofs of the metric dimension of $ \ZZ_c^2 $ and $ \ZZ_c^3 $.
We believe that our methods can furthermore be applied to Static Black-Peg 
Mastermind for constant $p>3$, and also to the situation where the pegs have
independent numbers of colors. 
This would lead to new discoveries regarding the metric dimension of $ 
\ZZ_c^p$ for $p>3$ and of $ \ZZ_{c_1} \times \ZZ_{c_2} \times \times \ZZ_{c_3} $, 
where $ c_1, c_2, c_3 $ may differ.


\bibliographystyle{plain}
\bibliography{static_abgame_23_20}

\end{document}